\documentclass[aos]{imsart}
\usepackage{appendix}
\usepackage{mathtools}
\usepackage[T1]{fontenc}
\usepackage{amsfonts}
\usepackage{amsmath}
\usepackage{amssymb}
\usepackage{amsthm}
\usepackage{bbm}
\usepackage{bm}
\usepackage{mathrsfs}
\usepackage{color}
\usepackage{pdfsync}
\usepackage{enumitem}
\usepackage{hyperref}

\usepackage{graphicx}
\usepackage{subfigure}
\usepackage{tikz}

\DeclareMathOperator*{\OT}{OT}
\newcommand{\ROT}{\displaystyle{\operatorname{ROT}}}
\DeclareMathOperator*{\KL}{KL}
\newcommand{\DUAL}{\displaystyle{\operatorname{DUAL}}}

\DeclareMathOperator*{\argmax}{argmax}

\DeclareMathOperator*{\Var}{Var}

\newcommand{\RR}{\mathbb{R}}
\newcommand{\R}{\RR}
\newcommand{\Z}{\mathbb{Z}}

\newcommand{\PP}{\mathbb{P}}
\newcommand{\NN}{\mathbb{N}}

\newcommand{\eps}{ \varepsilon}

\newcommand{\Banach}{\mathcal{C}_\oplus}
\newcommand{\as}{\text{-a.s.}}
\usepackage[capitalize]{cleveref}

\theoremstyle{plain}
\newtheorem{theorem}{Theorem}[section]
\newtheorem{proposition}[theorem]{Proposition}
\newtheorem{lemma}[theorem]{Lemma}
\newtheorem{corollary}[theorem]{Corollary}
\theoremstyle{definition}
\newtheorem{definition}[theorem]{Definition}
\newtheorem{remark}[theorem]{Remark}
\newtheorem{example}[theorem]{Example}
\newtheorem{assumption}[theorem]{Assumption}
\theoremstyle{remark}

\begin{document}

\begin{frontmatter}
\title{Sparse Regularized Optimal Transport\\without Curse of Dimensionality}

\runtitle{Sparse Regularized Optimal Transport}
\runauthor{A.~Gonz\'alez-Sanz, S.~Eckstein, and M.~Nutz}
\begin{aug}

\author[A]{\fnms{Alberto} \snm{Gonz\'alez-Sanz}\ead[label=e1]{ag4855@columbia.edu}},
\author[B]{\fnms{Stephan} \snm{Eckstein}\ead[label=e2]{stephan.eckstein@uni-tuebingen.de}} \and
\author[C]{\fnms{Marcel} \snm{Nutz}\ead[label=e3]{mnutz@columbia.edu}}

\address[A]{Department of Statistics, Columbia University  \printead{e1}}
\address[B]{Department of Mathematics, University of Tübingen \printead{e2}}
\address[C]{Departments of Statistics and Mathematics, Columbia University \printead{e3}}
~ \\

\end{aug}

\begin{abstract}
Entropic optimal transport---the optimal transport problem regularized by KL diver\-gence---is highly successful in statistical applications. Thanks to the smoothness of the entropic coupling, its sample complexity avoids the curse of dimensionality suffered by unregularized optimal transport. The flip side of smoothness is overspreading: the entropic coupling always has full support, whereas the unregularized coupling that it approximates is usually sparse, even given by a map. Regularizing optimal transport by less-smooth $f$-divergences such as Tsallis divergence (i.e., $L^p$-regularization) is known to allow for sparse approximations, but is often thought to suffer from the curse of dimensionality as the couplings have limited differentiability and the dual is not strongly concave. We refute this conventional wisdom and show, for a broad family of divergences, that the key empirical quantities converge at the parametric rate, independently of the dimension. More precisely, we provide central limit theorems for the optimal cost, the optimal coupling, and the dual potentials induced by i.i.d.\ samples from the marginals. These results are obtained by a powerful yet elementary approach that is of broader interest for Z-estimation in function classes that are not Donsker. %
\end{abstract}

\end{frontmatter}

\section{Introduction}

Optimal transport has become ubiquitous following computational advances enabling applications in statistics, machine learning, image processing, and other domains where distributions or data sets
need to be compared (e.g., \cite{PeyreCuturi.19, villani2008optimal}). Given probability measures~$P$ and~$Q$ on~$\mathbb{R}^d$, and a cost function $c:\mathbb{R}^d\times\mathbb{R}^d\to\mathbb{R}$, the Monge--Kantorovich optimal transport problem is
\begin{equation}\label{otIntro}
  \OT(P,Q)=\inf_{\pi\in \Pi(P,Q)} \int c(x,y) d\pi(x,y)
\end{equation} 
where $\Pi(P,Q)$ denotes the set of couplings; i.e., joint distributions $\pi$ on $\mathbb{R}^d\times\mathbb{R}^d$ with marginals $P$ and $Q$. A key bottleneck for statistical applications is that optimal transport suffers from the curse of dimensionality in terms of sample complexity (see \cite{Chewi_Niles-Weed_Rigollet,Dudley.69,Fournier.2014.PTRF} among many others). More specifically, let  \(X_1, \dots, X_n\) and \(Y_1, \dots, Y_n\) be i.i.d.\ samples from $P$ and $Q$, respectively, and consider their empirical measures $P_{n}=\frac{1}{n}\sum_{i=1}^{n}\delta_{X_i} $ and $Q_{n}=\frac{1}{n}\sum_{i=1}^{n}\delta_{Y_i}$. Then $\OT(P_n,Q_n)$, the value of~\eqref{otIntro} computed with marginals $(P_n,Q_n)$ instead of $(P,Q)$, converges to the population value $\OT(P,Q)$ at a rate that deteriorates exponentially with the dimension $d$. For instance, for the important cost $c(x,y)=\|x-y\|^2$ defining the 2-Wasserstein distance, $\mathbb{E}[|\OT(P_n,Q_n) - \OT(P,Q)|] \sim n^{-2/d}$ for dimension $d\geq 5$, under regularity conditions on the population marginals $P$ and $Q$ (see \cite{ManoleNilesWeed.24} for this particular result).

By contrast, the celebrated entropy-regularized optimal transport (EOT) avoids the curse of dimensionality. The EOT problem penalizes~\eqref{otIntro} with the Kullback--Leibler (KL) divergence $\KL(\pi|P\otimes Q)$ between the coupling $\pi$ and the product $P\otimes Q$ of the marginals,
\begin{equation}\label{eotIntro}
  {\rm EOT}_{\eps}(P,Q):=  \inf_{\pi\in \Pi( P ,   Q )} \int c d\pi + \eps \KL(\pi|P\otimes Q)
\end{equation} 
where 
$\eps>0$ is a fixed parameter determining the strength of regularization. Indeed, a series of works starting with \cite{genevay.2019.PMLR} has shown that ${\rm EOT}_{\eps}(P_n,Q_n)$ converges to ${\rm EOT}_{\eps}(P,Q)$ at the parametric rate $n^{-1/2}$, thus the dimension $d$ (and the parameter $\eps$) only affect the constants. Literature is reviewed in \cref{se:literature} below. The prevailing thinking is that EOT overcomes the curse of dimensionality thanks to its smoothness. Indeed, the entropic penalty leads to optimal couplings whose density is smooth (or at least as smooth as the cost $c$), and this regularity holds independently of the marginals. In particular, it holds uniformly over the empirical measures $(P_n,Q_n)$, and this fact drives the aforementioned result  of \cite{genevay.2019.PMLR}.

The flip side of smoothness is that the optimal coupling of EOT always has full support (i.e., the same support as $P\otimes Q$), for any value of the regularization parameter $\eps>0$. By contrast, the unregularized optimal transport coupling that it approximates, typically has sparse support (the graph of a function, namely the Monge map). This disconnect can be undesirable depending on
the application; for instance, the large support (or ``overspreading'') can amount to blurrier images in an image processing task as shown
in~\cite{blondel18quadratic}, bias in a manifold learning task as in~\cite{zhang.2023.manifoldlearningsparseregularised}, or unfaithful approximation of barycenters \cite{LiGenevayYurochkinSolomon.2020.NIPS}. In such cases, sparse approximations are desirable.

It is known that less-smooth $f$-divergence penalties give rise to sparse approximations. Indeed, consider the divergence-regularized optimal transport problem with the divergence $D_\varphi(\pi|P\otimes Q)$ defined by a function $\varphi$,
\begin{align}
  {\rm ROT}_{\eps}(P,Q)&:=  \inf_{\pi\in \Pi( P ,   Q )} \int c d\pi  
  +{\eps } D_\varphi(\pi|P\otimes Q), \label{rotIntro}\\
  D_\varphi(\pi|P\otimes Q)&:= \int \varphi\left( \frac{d \pi}{ d( P \otimes   Q )}  \right) d( P \otimes   Q ), \label{rotIntroDivergence}
\end{align} 
where $\varphi:\mathbb{R}_+\to\mathbb{R}$ is convex and satisfies certain conditions (see \cref{as:divergence}). While the KL divergence of EOT is recovered for $\varphi(t)=t\log t$, replacing this by a power $\varphi(t)=(p-1)^{-1}(t^p-1)$ for $p>1$ (Tsallis divergence, including the quadratic or $\chi^2$ divergence for $p=2$) was studied starting with \cite{Muzellec.2017.AAAI,blondel18quadratic,EssidSolomon.18} and empirically seen to have sparse solutions. Recent investigations underline this finding with theoretical results; see the literature review below. The main objection against such regularizations is that they lack the smoothness of EOT. Indeed, the density of the optimal coupling is only as smooth as the derivative of the convex conjugate $\psi(s)=\sup_{t\geq 0} \{ st-\varphi(t)\}$ of $\varphi$, which for $\varphi(t)=(p-1)^{-1}(t^p-1)$ is $k$-times differentiable only for $k< p/(p-1)$. In particular, it does not enjoy the $\mathcal{C}^\infty$-smoothness crucially used in EOT, where the proof of sample complexity uses derivatives of higher and higher order as the dimension $d$ grows. Following the aforementioned prevailing thinking, it is often assumed that~\eqref{rotIntro} thus suffers from the curse of dimensionality when it is not smooth. This idea is in line with the recent work~\cite{BayraktarEckstein.2025.BJ} which gives an upper bound for the sample complexity that deteriorates exponentially with the dimension~$d$ (see \cref{se:literature}).

In the present paper, our goal is to refute this prevailing thinking. For a broad class of regularizations $\varphi$, we show that~\eqref{rotIntro} overcomes the curse of dimensionality and converges at the parametric rate, despite a lack of smoothness (and also of a PL inequality, cf.\ \cref{se:literature}). In fact, our results are much more detailed: we establish central limit theorems  for all key objects. As a consequence, \eqref{rotIntro} provides approximate solutions to optimal transport with precise statistical guarantees.

\paragraph*{Synopsis of main results} %
We consider marginals $P,Q$ with bounded supports, at least one of them connected. The transport cost $c$ is a general $\mathcal{C}^1$ function. The main assumption on the divergence is that the convex conjugate~$\psi$ of $\varphi$ is $\mathcal{C}^2$. This includes the KL divergence where $\psi$ is $\mathcal{C}^\infty$, but also Tsallis divergences that yield sparse approximations of unregularized optimal transport. There are three main objects under consideration. The first is the optimal cost $\ROT_\eps(P_n,Q_n)$; here the result is simple to state as the object is scalar.  In \cref{th:CLT.cost}, we show that the optimal cost is asymptotically normal at rate $\sqrt{n}$. More precisely, $ \sqrt{n}\big( \ROT(P_n,Q_n) -\ROT(P,Q) \big)\to N(0,\sigma^2)$ where $N(0,\sigma^2)$ is the centered normal distribution with a variance $\sigma^2$ detailed in \cref{th:CLT.cost}. The second key object is the optimal coupling $\pi_n\in\Pi(P_n,Q_n)$. Here the asymptotic normality can be stated by integrating a test function: \cref{th:CLT.plans} shows that for any bounded measurable function~$\eta$, we have $ \sqrt{n}\left(\int \eta d(\pi_n - \pi) \right) \to  N(0, \sigma^2(\eta))$, where again the variance is detailed in the theorem. The third object is the pair of dual potentials; that is, functions $(f_*,g_*)$ solving the dual problem of~\eqref{rotIntro} in the sense of convex analysis,
\begin{align}
    \label{dualIntro}
     \sup_{(f,g)\in L^\infty(  P )\times L^\infty(  Q )} \int\left\{ f(x) + g(y) - \eps\cdot\psi\left(\frac{f(x)+g(y)-c(x,y)}{\eps} \right) \right\}d(  P \otimes   Q )(x,y).
\end{align}
The potentials determine the optimal coupling via
$$
d\pi =\psi'\left(\frac{f_*(x)+g_*(y)-c(x,y)}{\eps}\right) d(P\otimes Q)
$$
where $\psi'$ is the derivative of the conjugate of $\varphi$. Unlike in the special case of EOT, $\psi'$ is generally not invertible, and hence the potentials are the main object in our study. Denoting by $(f_*,g_*)$ the population potentials and by $(f_n,g_n)$ the empirical counterparts, our  central limit theorem for the potentials states that 
$\sqrt{n} (f_n-f_*,g_n-g_*)$ converges (weakly wrt.\ the uniform norm) to the Gaussian random process detailed in \cref{th:CLT.pot}. Of course this result can only hold if $(f_*,g_*)$ is unique. Indeed, we obtain uniqueness (up to additive constants) as a by-product of our analysis; this is a new result of its own interest and holds even for merely continuous transport costs~$c$ (\cref{th:uniqueness}). 

\subsection{Related literature}\label{se:literature}

Next, we review literature on optimal transport regularized by an $f$-divergence. Following the main focus of the present work, we begin with divergences other than the entropic one. After that, we review the existing sample complexity results for EOT.
\subsubsection{Non-entropic divergences} Optimal transport problems regularized by Tsallis divergence were first considered in the discrete setting, starting with \cite{Muzellec.2017.AAAI} for applications in ecological inference. Focusing on the particular case of quadratic (i.e., $\chi^2$) divergence, \cite{blondel18quadratic} highlighted the sparsity of the solution for use in image processing and \cite{EssidSolomon.18} studied minimum-cost flow problems on graphs. The earlier work~\cite{DesseinPapadakisRouas.18} considered optimal transport with a more general convex regularization. In the continuous setting, divergence-regularized optimal transport was first explored in the computational literature. Several works including \cite{EcksteinKupper.21, genevay.2019.PMLR, GulrajaniAhmedArjovskyDumoulinCourville.17, LiGenevayYurochkinSolomon.2020.NIPS, seguy2018large} approach the dual problem by optimization techniques. For instance, \cite{LiGenevayYurochkinSolomon.2020.NIPS} computes regularized Wasserstein barycenters using neural networks. From a computational point of view, Tsallis divergences are attractive to mitigate a well-known issue of EOT for small values of the regularization parameter $\eps$, namely that optimization methods struggle with the exponentially small values of the density (e.g., \cite{Schmitzer.19}).

On the theoretical side, \cite{LorenzMannsMeyer.21} established duality results for the case of quadratic divergence. Also for quadratic divergence, \cite{LorenzMahler.22} showed convergence $\ROT_{\eps} \to \OT$ as the regularization parameter $\eps$ tends to zero. More recently, \cite{EcksteinNutz.22} derives a rate for this convergence, for general $f$-divergences. For the special case of quadratic divergence, \cite{GarrizmolinaElAl.2024} shows that this rate corresponds to the  exact leading order and identifies the multiplicative constant, whereas \cite{GonzalezSanzNutz.2024.Quantita} focuses on the discrete case where convergence is stationary. The recent works \cite{GonzalezSanzNutz2024.Sacalar,WieselXu.24} study the size of the optimal support in the regime $\eps\to0$, thus quantifying the sparsity of the solution (qualitatively proved in~\cite{Nutz.24} but empirically known long before, as mentioned above).

The only previous work on sample complexity for non-entropic divergences is \cite{BayraktarEckstein.2025.BJ}. The authors bound the expected absolute difference between the population and empirical optimal costs. The bound depends on the smoothness and the dimension of the problem. Recall that $\psi$ denotes the conjugate of the convex function $\varphi$ defining the divergence in~\eqref{rotIntroDivergence}. Let $k\in\NN$ be such that 
$\psi\in\mathcal{C}^k$ as well as $c\in\mathcal{C}^k$. Moreover, let $d$ be the dimension of the marginals. %
The main result \cite[Theorem 5.3]{BayraktarEckstein.2025.BJ} states that for $d>2k$,
\begin{equation}\label{eq:looseBound}
  \mathbb{E}[|\ROT_\eps(P,Q)-\ROT_\eps(P_n,Q_n)|]\lesssim n^{-\frac{k}{d}}.
\end{equation}
For most applications, the bottleneck in this bound is the smoothness of $\psi$, rather than~$c$. Assuming merely $\psi\in\mathcal{C}^2$ as in the present work, the rate~\eqref{eq:looseBound} is  $n^{-\frac{2}{d}}$ for $d\geq5$, suggesting that the curse of dimensionality is exactly \emph{the same as for the unregularized} optimal transport problem mentioned at the beginning of this introduction. The bound~\eqref{eq:looseBound} is obtained by adapting the classical approach going back to \cite{genevay.2019.PMLR} for EOT: estimating the regularity of the empirical potentials to define a function space with controlled covering number and then applying empirical process theory. Because the potentials are only as smooth as~$\psi'$, this yields a bound~\eqref{eq:looseBound} that in general deteriorates exponentially with the dimension (whereas when $\psi\in\mathcal{C}^\infty$ as for EOT, the dimension can be eliminated).

The present work indicates that the bound~\eqref{eq:looseBound} is loose, even to the extent that its implicit message regarding the curse of dimensionality is misleading. Indeed, we show that the actual rate is $n^{-\frac{1}{2}}$, \emph{independently of the dimension $d$,} and in particular that the influence of the dimension is more similar to EOT than to unregularized optimal transport. One may consider this result highly surprising given the lack of smoothness (and failure of strong concavity, cf.\ \cref{se:introEntropic}). As we derive central limit theorems with the usual $\sqrt{n}$ scaling for all key quantities, this rate is sharp. (Except possibly in degenerate examples where the Gaussian limit has vanishing variance; then, the actual rate is likely faster.)

\subsubsection{Entropic divergence}\label{se:introEntropic} The literature on EOT is extremely large, and fairly well known. We only review the literature on sample complexity, focusing on continuous population marginals (and with a view towards explaining why the existing approaches fail in our setting). %
As mentioned above, parametric rates for the optimal EOT cost~\eqref{eotIntro} were first obtained by \cite{genevay.2019.PMLR}. Assuming a smooth (i.e, $\mathcal{C}^\infty$) transport cost $c$ and compactly supported marginals, the authors show that the potentials are smooth with $\mathcal{C}^k$ norm bounded independently of the marginal measures, for any $k\in\NN$. Thus the empirical potentials all belong to a function class with controlled covering number. The result then follows by applying empirical process theory to the dual problem of EOT. A similar bound with improved constants and more general (sub-Gaussian) marginals was obtained in \cite{MenaWeed.2019.Nips} using a refinement of the same approach. Moreover, \cite{MenaWeed.2019.Nips} provided the first central limit-type theorem on EOT, namely for the optimal cost. In this result, the centering is by the mean of the empirical cost, instead of the population cost as in a usual central limit theorem. The proof follows an approach based on the Efron--Stein inequality that is adapted from (unregularized) optimal transport; cf.~\cite{delBarrioLoubes.2019.AOP}. A central limit theorem for the EOT cost in the classical sense (centered at the population value) was first derived in \cite{delBarrioEtAl.2023.SIMODS}. The proof combines the result of \cite{MenaWeed.2019.Nips} with an analysis of the bias--variance decomposition showing that the bias tends to zero faster than $n^{-1/2}$. This is based on the fact that the empirical EOT potentials belong to a Donsker class. Namely, the $\mathcal{C}^k$-norm of the potentials is uniformly bounded for any $k$, and choosing $k>d/2$ implies the Donsker property~\cite[p.\,157]{vanderVaart.1996}. The authors also show that the potentials converge at the parametric rate (in the norm of $\mathcal{C}^{k}$, for any $k\in \NN$); this is based on the strong convexity of the dual problem (see \cite[Lemma~4.6]{delBarrioEtAl.2023.SIMODS}). In concurrent work, \cite{goldfeld.2024.statisticalinferenceregularizedoptimal} provided a very similar central limit theorem for the EOT cost, using a different proof technique based on the functional delta method for supremum-type functionals (see \cite{Carcamo.Cuevas.Rodrigez.2020.BJ}). This argument again rests on the potentials belonging to a Donsker class. 

While the aforementioned central limit theorems were all about the optimal cost, central limit theorems for the potentials and couplings were first derived in  \cite{GonzalezSanz.2024.weaklimits} and \cite{Goldfeld.2024.EJS}. Both works use the same methodology of Z-estimation and the delta method (see also \cref{se:methodology} below). Specifically, \cite[Theorem 3]{Goldfeld.2024.EJS} shows that the potentials are Hadamard differentiable as functions of the marginal measures, tangentially to perturbations in $\mathcal{C}^k(\Omega)'$ for any $k\in \NN$. As the unit ball in  $\mathcal{C}^k(\Omega)$ is a Donsker class for $k>d/2$, the delta method then yields the central limit theorem for the potentials, and the theorems for the couplings are derived from there.

The aforementioned works all deal with a transport cost function $c\in\mathcal{C}^\infty$, and exploit the fact that the EOT potentials are as smooth as that cost. A substantially different approach was taken in \cite{rigollet2022samplecomplexityentropicoptimal}. Assuming only that~$c$ is bounded, the approach is based on the fact that the dual problem of EOT is strongly concave, and hence satisfies a PL inequality, uniformly over the empirical marginals (see \cite[Lemma~16]{rigollet2022samplecomplexityentropicoptimal}). The authors then obtain the parametric rate for the convergence of the empirical EOT cost; more precisely, the mean squared error and bias are both bounded by $1/n$, with constants that are fully dimension-independent.
In the setting of non-smooth cost, central limit theorems for the optimal costs and the optimal couplings were established in \cite{GonzalezSanz.2023.Beyond}. The authors linearize the potentials in the empirical $L^2$ norms and, unlike \cite{Goldfeld.2024.EJS,GonzalezSanz.2024.weaklimits}, do not use empirical processes theory %
but instead approximate the optimal couplings by infinite order U-statistics. This approximation holds by a uniform contraction argument over the linearized Schrödinger system, which is deeply related to the PL inequality in \cite{rigollet2022samplecomplexityentropicoptimal}.

All of the above approaches fail in our setting because the empirical potentials are not smooth (and do not belong to a Donsker class) while the dual problem is not strongly concave and does not satisfy a uniform PL inequality.

While our main interest is in divergences other than the entropic one, let us mention that we add to the literature even in the special case of EOT. Namely, we provide a central limit theorem for the potentials under the assumption that the cost $c$ is $\mathcal{C}^1$, where such a result previously existed only for $c\in\mathcal{C}^\infty$ (see \cite{Goldfeld.2024.EJS,GonzalezSanz.2024.weaklimits}).  Our proof technique is substantially different, as we explain next.

\subsection{Methodology}\label{se:methodology}

We follow the approach of Z-estimation in deriving central limit theorems; see \cite[Chapter~3.3]{vanderVaart.1996}. While the usual argument via Donsker classes is doomed to fail due to the missing regularity of the empirical potentials, one key methodological innovation is to overcome this issue with a novel line of argument. As our approach may be useful for Z-estimation problems in other areas, we sketch the approach in general terms.

In Z-estimation, the empirical quantities $\theta_n$ of interest are described by an equation of the form $\Phi_n(\theta_n)=0$ with a random operator $\Phi_n$ while the population counterpart $\theta_*$ is described by $\Phi(\theta_*)=0$ with a deterministic $\Phi$.  The goal is to show a central limit theorem for the convergence $\theta_n\to\theta_*$. More specifically, $\theta_n,\theta_*$ are elements of the parameter set $\Theta$ which is contained in a Banach space $(\mathcal{B}, \|\cdot\|)$. We assume for simplicity that $\Phi_n,\Phi: \mathcal{B}\to\mathcal{B}$; in general the image may be in another Banach space. In our particular problem, $\theta_n$ are the empirical potentials $(f_n,g_n)$ and $\theta_*$ is the population counterpart $(f_*,g_*)$. The nonlinear operators $\Phi_n,\Phi$ represent the first-order conditions of optimality in the dual problem. As in many applications of Z-estimation, these operators have the integral form
\begin{align}\label{eq:integralFunctionals}
\Phi_n(\theta)= \int \phi(\theta) d \mathbb{P}_n, \qquad  \Phi(\theta)= \int \phi(\theta) d \mathbb{P} 
\end{align}
where $\phi $ is deterministic and $\mathbb{P}_n$ are empirical measures derived from i.i.d.\ samples of $\mathbb{P}$.

The basic theorem of Z-estimation (see \cite[Theorem~3.3.1]{vanderVaart.1996}) separates the conditions for the desired asymptotic normality of $\sqrt{n}(\theta_n-\theta_*)$ into analytical conditions on the population quantities $\theta_*,\Phi$ and a stochastic  condition stating that the remainder in a Taylor expansion is negligible. Roughly, these conditions are 

\begin{enumerate}
    \item $[\Phi_n-\Phi](\theta_*)$ satisfies a central limit theorem in $\mathcal{B}$,
    \item $\Phi$ is Fr\'{e}chet differentiable at $\theta_*$ with invertible derivative $\mathbb{L}:=D\Phi(\theta_*)\in\mathcal{L}(\mathcal{B},\mathcal{B})$,
    \item $\|\theta_n-\theta_*\| \xrightarrow{\PP} 0$,
    \item 
    the following expansion holds,
\begin{equation}\label{eq:ZestimationTaylor}
\Delta_n :=[\Phi_n-\Phi](\theta_n) - [\Phi_n-\Phi](\theta_*) = o_\mathbb{P}\big(n^{-1/2}+\|\theta_n-\theta_*\|\big).
\end{equation}
\end{enumerate}
While obtaining invertibility has its own challenges in our setting, we defer that discussion to \cref{se:proof.strategy}. Like in most applications, the main issue is to establish~\eqref{eq:ZestimationTaylor}. The standard approach (see \cite[Lemma~3.3.5]{vanderVaart.1996}) is to use a sufficient condition whose main part is that the random functions $\phi(\theta_n)-\phi(\theta_*)$ form a $\mathbb{P}$-Donsker  class (for $\|\theta_n-\theta_*\|$ sufficiently small). In the setting of EOT, the smoothness of the empirical potentials indeed yields the Donsker property, and together with the aforementioned Hadamard differentiability, this allowed \cite{Goldfeld.2024.EJS,GonzalezSanz.2024.weaklimits} to infer the desired central limit theorem for the potentials. In the present case, this approach is a nonstarter because the regularity of the potentials is too poor to give a Donsker class. 

We take a different route. %
While our aim is, as before, a central limit theorem in $\mathcal{B}$, we also use an auxiliary Banach space $\mathcal{B}_s\subset \mathcal{B}$ equipped with a stronger norm $\|\cdot\|_s$ such that the unit ball of $\mathcal{B}_s$ is compactly embedded in $\mathcal{B}$. In our case, $\mathcal{B}$ is (up to some details) the space of continuous functions on a compact set endowed with the uniform norm whereas $\mathcal{B}_s$ is (up to an isomorphism) the space of H\"{o}lder continuous functions with a fixed H\"{o}lder exponent $\beta\in (0,1)$. The stronger H\"{o}lder norm guarantees the compact embedding. While we keep the conditions (ii), (iii) above, we strengthen (i) to 
\begin{enumerate}
    \item[(i')] $[\Phi_n-\Phi](\theta_*)$ satisfies a central limit theorem in $\mathcal{B}_s$
\end{enumerate}
as the stronger topology will be used to help establish~\eqref{eq:ZestimationTaylor}. We verify (i') by applying a general central limit theorem for H{\"o}lder spaces detailed in~\cref{se:HolderCLT}. The key intermediate step towards~\eqref{eq:ZestimationTaylor} is that
\begin{enumerate}
    \item[(iv')] there are a compact $\mathcal{K}\subset \mathcal{B}$, random elements $U_n\in\mathcal{K}$, and random elements $V_n\in\mathcal{B}$ with  $\|V_n\|\xrightarrow{{\mathbb{P}}} 0$,  such that 
    $$ \Delta_n :=[\Phi_n-\Phi](\theta_n) - [\Phi_n-\Phi](\theta_*)  =(U_n +V_n)\| \theta_n-\theta_*\|.$$
\end{enumerate}
This is verified using the specifics of the problem at hand (cf.\ \cref{lemma:develomentIncompact}), without requiring smoothness. Next, we outline how these conditions lead to~\eqref{eq:ZestimationTaylor}. Note that the differentiability condition (ii) implies 
$$ (\theta_*-\theta_n)= \mathbb{L}^{-1}(\Phi (\theta_*)-\Phi(\theta_n)) + R_n $$
with $\|R_n\|=o_{\mathbb{P}}(\|\theta_*-\theta_n\|)$. Using $\Phi (\theta_*)=0=\Phi_n (\theta_n)$, this can be written as
\begin{align*}
    \theta_*-\theta_n&= \mathbb{L}^{-1}(\Phi_n (\theta_n)-\Phi(\theta_n)) + R_n
    =\mathbb{L}^{-1}\big(\Delta_n + [\Phi_n-\Phi] (\theta_*) \big)+ R_n.
\end{align*}
By (i'), the sequence $\sqrt{n}[\Phi_n-\Phi](\theta_*)$ is tight in $\mathcal{B}_s$. Using also (iv') and the continuity of $\mathbb{L}^{-1}$, we deduce a key decomposition (cf.\ \cref{pr:three.terms})
\begin{align}\label{eq:threeTermsIntro}
    \theta_*-\theta_n= \| \theta_n-\theta_*\| (U_n'+V_n')+ n^{-\frac{1}{2}}{W}_n
\end{align}
with random functions $U'_n$ taking values in a compact $\mathcal{K}'\subset\mathcal{B}$, random functions $\|V_n'\|\xrightarrow{\mathbb{P}} 0$, and $(\mathbb{L}W_n)$ tight in $\mathcal{B}_s$. From here, we make (mild) use of the particular properties the functionals~\eqref{eq:integralFunctionals} at hand. Namely, $\phi$ is differentiable, hence the  fundamental theorem of calculus allows us to write
\begin{align*}
  \Delta_n &=\int   (\phi(\theta)- \phi(\theta_n))d (\mathbb{P}- \mathbb{P}_n)
  =\int   (\theta_*-\theta_n) \phi_n  d(\mathbb{P}- \mathbb{P}_n)
\end{align*}
for $\phi_n:=\int_0^1 \phi'(\lambda \theta_*+ (1-\lambda)\theta_n )d \lambda$. Inserting~\eqref{eq:threeTermsIntro} yields
\begin{align*}
  \Delta_n &=\| \theta_*-\theta_n\| \int (U_n'+ V_n')   \phi_n  d(\mathbb{P}- \mathbb{P}_n)+ n^{-\frac{1}{2}}\int W_n   \phi_n  d(\mathbb{P}- \mathbb{P}_n).
\end{align*}
The stated properties of $U_n', V_n'$ enable us to verify that first term is $o_\mathbb{P}\big(\|\theta_n-\theta_*\|\big)$. Similarly, the tightness of $(\mathbb{L}W_n)$ in $\mathcal{B}_s$ and the fact that $\|\cdot\|_s$-bounded sets are relatively compact in $\mathcal{B}$ allow us to verify that the second term is $o_\mathbb{P}\big(n^{-1/2}\big)$. This completes the derivation of~\eqref{eq:ZestimationTaylor}. The central limit theorem for the potentials then follows by standard arguments. The proofs of the central limit theorems for the optimal cost and the optimal couplings are based on the result for the potentials, \eqref{eq:ZestimationTaylor} and the central limit theorem for two-sample $U$-statistics.

Finally, we comment on the uniqueness of the population potentials, shown in  \cref{th:uniqueness} when the cost $c$ is continuous and one marginal has connected support. The usual technique to obtain uniqueness---familiar in unregularized optimal transport and adapted in \cite{Nutz.24} for the regularized problem with quadratic divergence---is to argue that the gradient of the potential equals the partial derivative of the cost $c$ on the support of the optimal coupling. 
This argument is a nonstarter if the cost $c$ is not differentiable. In the present work, uniqueness is instead deduced from the invertibility of $\mathbb{L}$. Clearly uniqueness of the population potentials is a precondition for any result on the convergence to them. On the other hand, we mention that the empirical potentials are in general not unique. They are unique in the special case of EOT and more generally when the optimal coupling has sufficiently large support, but not in general. The family of all potentials can be parametrized by the components (in the sense of connectedness of graphs) of the support, as shown in~\cite{Nutz.24}. 

\paragraph*{Organization} The remainder of this paper is organized as follows. \Cref{se:prelims} details the problem formulation and notation, and summarizes basic facts about the regularized optimal transport problem. \Cref{se:mainRes} first states the main results, namely the uniqueness of the population potentials and the three central limit theorems, then continues with an overview of the proof strategy. The proofs start in \cref{se:linearization} by analyzing the linearization of the first-order condition of the potentials, here the key result is the invertibility of the derivative~$\mathbb{L}$ (\cref{pr:Invertible}). The proofs of the central limit theorems are presented in \cref{se:proofs.of.CLTs}. We first show in \cref{se:consistency} that the empirical potentials are a consistent estimator. \Cref{se:key.decomposition} contains the technical core of the proof, namely the aforementioned decomposition into three terms, each with a different compactness property. \Cref{se:proof.CLT.end} concludes with the proofs of the central limit theorems. In  \cref{sec:simulations} we provide a numerical analysis of a particular example where the population problem can be solved in closed form and hence the convergence rate can be accurately observed. \Cref{se:HolderCLT} provides an abstract central limit theorem for the Banach space of H\"{o}lder functions that is used in the proof of our main results. Omitted proofs are collected in \cref{se:proofsAppendix}.

\pagebreak[3]

\section{Problem statement and preliminaries}\label{se:prelims}

\subsection{Notation}\label{Notation}

Let $\Omega$ be a compact subset of $\mathbb{R}^d$. For $f: \Omega \to \mathbb{R}$, we write 
$\|f\|_\infty = \sup_{x \in \Omega} |f(x)|$. 
The space of continuous functions $\Omega\to\mathbb{R}$ is denoted $\mathcal{C}(\Omega)$ and endowed with $\|\cdot\|_\infty$. Moreover, $\mathcal{C}^k(\Omega)$ denotes $k$-times continuously differentiable functions. For $\alpha \in (0, 1]$, the $\alpha$-H\"{o}lder semi-norm of $f: \Omega \to \mathbb{R}$ is 
\[
[f]_{\alpha} = \sup_{x \neq y} \frac{|f(x) - f(y)|}{\|x - y\|^\alpha}.
\]  
We set $\mathcal{C}^{0,\alpha}(\Omega)=\{f\in\mathcal{C}(\Omega):[f]_{\alpha}<\infty\}$, a Banach space under the norm  
$
\|f\|_{0,\alpha} := \|f\|_\infty + [f]_\alpha.
$  
If $\mathcal{B}_1$ and $ \mathcal{B}_2$ are Banach spaces, the space of bounded linear operators $F: \mathcal{B}_1\to \mathcal{B}_2$ is denoted $\mathcal{L}(\mathcal{B}_1,\mathcal{B}_2)$ and endowed with the norm topology. 

The open unit ball in $\mathbb{R}^d$ is denoted by $\mathbb{B}$  and its closure by $\overline{\mathbb{B}}$. Thus $x + r \, \mathbb{B}$ is the ball centered at~$x$ with radius~$r$. The gradient of $f:\R^d\to \R$ is denoted $\nabla f$ and the gradients of a function $c:\R^d\times \R^d\to \R$ with respect to  the first and second coordinates are denoted $\nabla_x c$ and $\nabla_y c$, respectively. The derivative of a univariate function  $\psi:\R\to \R$ is denoted $\psi'$. Given functions $x\mapsto f(x)$ and $y\mapsto g(y)$, we denote by $f\oplus g$ the function $(x,y)\mapsto f(x)+ g(y)$.

We fix a probability space $(\boldsymbol{\Omega}, \mathcal{A}, \mathbb{P})$ where all random variables are defined. For a measurable function $f: \mathbb{R}^d \to \mathbb{R}$ and a random vector $X: \boldsymbol{\Omega} \to \mathbb{R}^d$ with distribution $P$, the expectation of $f(X)$ is denoted  
$
\mathbb{E}[f(X)] = \int f(x) \, dP(x) = \int f \, dP.
$  
Almost-sure convergence is denoted by $\xrightarrow{\text{a.s.}}$, convergence in probability by $\xrightarrow{\mathbb{P}}$, and convergence in distribution (or weak convergence) by $\xrightarrow{w}$. The latter refers to convergence induced by continuous and bounded test functions. For scalar random variables $X_n,Y_n$, we write 
$X_n=\mathcal{O}_{\mathbb{P}}(Y_n)$ if $X_n/Y_n$ is stochastically bounded and 
$X_n = o_{\mathbb{P}}(Y_n)$ if $X_n/Y_n\xrightarrow{\mathbb{P}}0$.

More terminology, related to probability in Banach spaces, can be found in \cref{se:HolderCLT}. 

\subsection{Divergence}

Given measures $\mu,\nu$ on the same space, the divergence $D_\varphi(\mu|\nu)$ is determined by the function $ \varphi:[0, \infty)\to \R$ via
$$
   D_\varphi(\mu|\nu) = \int \varphi\left(\frac{d\mu}{d\nu}\right)d\mu,
$$
with the convention that $D_\varphi(\mu|\nu)=\infty$ if $\mu\not\ll \nu$. 
The following assumption on $\varphi$ is in force throughout the paper.

\begin{assumption}[Divergence]\label{as:divergence} The function $ \varphi:[0, \infty)\to \R$ is strictly convex with $\varphi(1) = 0 $,  $ \lim_{x\to \infty}\varphi(x)/x=+\infty $ 
and such that the conjugate $$y\mapsto \psi(y):=\varphi^*(y):= \sup_{x\geq 0} \{ xy-\varphi(x)\} $$ is in $\mathcal{C}^{2}(\R )$. Moreover, there exists $C > 0$ such
that $\psi'(x) \geq x$ for $x\geq C$, and there exist $t_0>0$ and $ \delta>0$ such that $\psi'(t_0)=1$ and $\psi$ is strictly convex on $ [t_0-\delta, \infty)$.
\end{assumption}

The detailed assumptions about the shape of $\psi$ are useful to derive basic estimates for the potentials; cf.\ \cref{pr:ROTprelims}. While they are a bit clumsy, they are verified in all examples of our interest. The more restrictive assumption is that $\psi$ is $\mathcal{C}^{2}$. 

\begin{example}\label{ex:divergences}

For the Kullback--Leibler divergence of EOT, we take $\varphi(x)=x\log(x)$ which yields $\psi(y)=e^{y-1}$. Note that $\psi\in\mathcal{C}^{\infty}$ and $\psi'(y)>0$ for all $y$, corresponding to the fact that the optimal coupling always  has full support. Next, consider the polynomial (Tsallis) divergence 
$$\varphi(x)=
\frac{x^\alpha-1}{\alpha}$$ for $\alpha\in(1,\infty)$. Then $\psi'(y)=y_+^{\beta-1}$ where $\frac{1}{\alpha} + \frac{1}{\beta}=1$ and $y_+=\max\{y,0\}$. This truncation at zero corresponds to the fact that the optimal coupling does not have full support in general. Note that $\psi$ is strictly convex on  $\R_+$ and in $\mathcal{C}^1$. Moreover,  $\psi$ is $\mathcal{C}^2$ for $\beta>2$; i.e., for any $\alpha\in (1,2)$. 

We note that the quadratic case $\alpha=2$ is not covered. For this boundary case, we expect to provide similar results as in the present paper in separate work, assuming that $c$ is the quadratic cost and the marginals have additional regularity (see also \cref{sec:simulations} for numerical hints). That setting implies additional convexity properties for the potentials. In the present work, we focus on giving a general result for a broad class of divergences, costs and marginals, thus showing that the conclusions do not hinge on a particular algebraic structure.
\end{example}

\subsection{Regularized optimal transport}

For brevity, we treat the regularized optimal transport problem for regularization parameter $\eps=1$. The general case is recovered by a simple scaling argument detailed in \cref{rk:general.regularization.parameter} below. Thus the primal problem is 
\begin{equation}\label{eq:primal1}
  \ROT(P,Q):=  \inf_{\pi\in \Pi( P,Q)} \int c d\pi  +  \int \varphi\left( \frac{d \pi}{ d( P \otimes   Q )}  \right) d( P \otimes   Q ).
\end{equation} 
The associated dual problem is 
\begin{equation}\label{eq:dual1}
  \DUAL(P,Q):=  \sup_{(f,g)\in  L^\infty(P) \times L^\infty(Q)} \, \int \big(f\oplus g - \psi(f\oplus g - c)\big) \,d(P\otimes Q)
\end{equation} 
and its first-order condition for optimality is
\begin{equation}
        \label{eq:FOC.as}
        \begin{cases}
             \int  \psi'(f_*(\cdot)+ g_*(y)-c(\cdot,y)) dQ(y)=1 \quad P\as,\\
             \int \psi'(f_*(x)+ g_*(\cdot)-c(x,\cdot)) dP(x)=1  \quad Q\as
        \end{cases}
\end{equation}

The following proposition summarizes some basic facts to be used throughout the paper. While the proofs are essentially known and largely contained in \cite{BayraktarEckstein.2025.BJ}, we give a self-contained and simpler proof in \cref{se:proofsPrelims}, for the convenience of the reader and also to rectify minor inaccuracies in~\cite{BayraktarEckstein.2025.BJ}. 

\begin{proposition}\label{pr:ROTprelims}
    Let $P,Q$ be probability measures on $\mathbb{R}^d$ with compact supports $\Omega,\Omega'$. Moreover, let $c\in\mathcal{C}(\Omega\times\Omega')$.
    \begin{enumerate}
    \item
    The strong duality 
      $\ROT(P,Q)=\DUAL(P,Q)$
    holds. 
    \item
    The primal problem~\eqref{eq:primal1} admits a unique optimizer $\pi\in \Pi( P,Q)$.
    \item
    The dual problem~\eqref{eq:dual1} admits a (non-unique) optimizer $(f_*,g_*)\in L^\infty(P)\times L^\infty(Q)$.
    \item
    A pair $(f_*,g_*)\in L^\infty(P)\times L^\infty(Q)$ is an optimizer of the dual problem~\eqref{eq:dual1} if and only if it satisfies the first-order condition~\eqref{eq:FOC.as}.
    \item
    Any dual optimizer $(f_*, g_*)$ yields a primal optimizer via
    $
    d\pi := \psi'(f_* \oplus g_* - c) d(P \otimes Q).    
    $
    \item Let $(f_*,g_*)\in L^\infty(P)\times L^\infty(Q)$ satisfy~\eqref{eq:FOC.as}. Then one can choose a version of $(f_*, g_*)$ satisfying \eqref{eq:FOC.as} everywhere; that is,
    \begin{equation}
        \label{eq:FOC.pw}
        \begin{cases}
             \int  \psi'(f_*(x)+ g_*(y)-c(x,y)) dQ(y)=1\quad \text{for all }x\in\Omega,\\
             \int \psi'(f_*(x)+ g_*(y)-c(x,y)) dP(x)=1\quad \text{for all }y\in\Omega'.
        \end{cases}
    \end{equation}
    For such a version $(f_*, g_*)$, the uniform bound $\|f_*\oplus g_* \|_\infty \leq 5 \|c\|_\infty + (\psi')^{-1}(1) < \infty$ holds.
    Moreover, if $\rho$ is a modulus of uniform continuity for~$c$ on $\Omega\times\Omega'$, then $f_*$ and $g_*$ are also uniformly continuous with modulus $\rho$.

    \item Let $(f_*, g_*)$ solve \eqref{eq:FOC.pw} and abbreviate $\xi_*(x,y):=f_*(x)+ g_*(y)-c(x,y)$. 
    Then there exists $\delta > 0$ such that
     \begin{align*}
        \int \psi''(\xi_*(\cdot,y))  dQ(y)\geq \delta \quad\mbox{on }\Omega, \qquad
        \int \psi''(\xi_*(x,\cdot))  dP(x)\geq \delta\quad\mbox{on }\Omega'.
     \end{align*}
\end{enumerate}
\end{proposition}

Functions $(f_*,g_*)$ that are optimal for the dual problem~\eqref{eq:dual1}, or equivalently solve the first-order condition~\eqref{eq:FOC.as}, will be called \emph{potentials} of the regularized optimal transport problem.

\begin{remark}[Limited Smoothness]
  \Cref{pr:ROTprelims}\,(vi) may suggest that the potentials~$f_*$ and~$g_*$ are ``as smooth as the cost~$c$.'' While this holds for the KL divergence of EOT, $\psi$ is in general only $\mathcal{C}^2$ for the polynomial divergences included in our setting, and this implies that $f_*$ and $g_*$ are only $\mathcal{C}^1$ in general. Indeed, the formula
  \begin{align*}   \nabla f_*(\cdot)&=\frac{\int \psi''(f_*(\cdot)+ g_*(y)-c(\cdot,y)) \nabla_x c(\cdot, y) dQ(y)}{\int \psi''(f_*(\cdot)+ g_*(y)-c(\cdot,y))  dQ(y)}
  \end{align*}
  highlights how the smoothness of both $\psi$ and $c$ affects the regularity of $f_*$. (See also \cite{Nutz.24} for a concrete counterexample.) As mentioned in the introduction, the limited smoothness of the empirical potentials  is a major obstacle for deriving our main results. %
\end{remark}

\begin{remark}[Regularization parameter]\label{rk:general.regularization.parameter}
  The regularized optimal transport problem is often considered with a regularization parameter $\eps>0$, which was taken to be $\eps=1$ above. In general,
  \begin{equation*}%
  \ROT_\eps(P,Q):=  \inf_{\pi\in \Pi( P,Q)} \int c d\pi  +  \eps\int \varphi\left( \frac{d \pi}{ d( P \otimes   Q )}  \right) d( P \otimes   Q ),
\end{equation*} 
the associated dual problem is 
\begin{equation}\label{eq:dual.eps}
    \DUAL_\eps(P,Q):=  \sup_{(f,g)\in  L^\infty(P) \times L^\infty(Q)} \, \int\left\{ f\oplus g - \eps\cdot\psi\left(\frac{f\oplus g(y)-c}{\eps} \right) \right\} \,d(P\otimes Q)
\end{equation} 
and the first-order condition for optimality reads
\begin{equation*}
        \begin{cases}
             \int  \psi'\left(\frac{f(\cdot)+ g(y)-c(\cdot,y)}{\eps}\right) dQ(y)=1,\\[.4em]
             \int \psi'\left(\frac{f(x)+ g(\cdot)-c(x,\cdot)}{\eps}\right) dP(x)=1.
        \end{cases}
\end{equation*}
The results for the general problem can be deduced from the special case $\eps=1$ by a simple scaling. Namely, defining $\bar{c}=c/\eps$, we have
$$\ROT_\eps(P,Q)=\eps {\ROT}(P,Q,\bar{c})$$ where ${\textstyle\ROT}(P,Q,\bar{c})$ is the problem with $\eps=1$ and cost $\bar{c}$. Moreover, the optimal couplings of these two problems coincide.  Similarly for the dual: $(f_\eps,g_\eps)$ is optimal for \eqref{eq:dual.eps} if and only if $(\eps f_\eps,\eps g_\eps)$ is optimal for the problem ${\DUAL}(P,Q,\bar{c})$ with $\eps=1$ and cost $\bar{c}$. For convenience, we detail in \cref{rk:general.regularization.parameter.main} below how the main results translate to a general  parameter $\eps>0$.
\end{remark}

\section{Main results}\label{se:mainRes}

Our main results are central limit theorems for the potentials, the optimal costs, and the optimal couplings. The following condition on the population marginals~$P$ and~$Q$ is imposed throughout this section.

\begin{assumption}[Marginals]\label{as:marginals}
The probability measures \( P \) and \( Q \) on $\R^d$ have compact supports \( \Omega \) and \( \Omega' \), respectively, and \( \Omega \) is connected.
\end{assumption}

A  central limit theorem for the potentials can only hold if the limiting (population) potentials $(f_*,g_*)$ are unique in a suitable sense. Indeed we obtain a uniqueness result under quite general conditions as by-product of our approach to the central limit theorem. Note that potentials always admit a degree of freedom: if $(f_*,g_*)$ are potentials, then so are $(f_*+a,g_*-a)$ for any $a\in\R$. Additionally, there is some freedom in changing potentials on nullsets, which can be mended by using continuous versions as in \cref{pr:ROTprelims}. To eliminate those degrees of freedom, we define the Banach space  $$\Banach=(\mathcal{C}(\Omega)\times \mathcal{C}(\Omega')/_{\sim_{\oplus}}$$  where $(f,g)\sim_{\oplus} (u,v)$ if $f\oplus g=u\oplus v$, endowed with the quotient norm 
$$ \|(f,g)\|_{\oplus}=\inf_{a\in \R}\{\|f-a\|_\infty+ \|g+a\|_\infty\}.$$
 The equivalence class of $(f,g)$ will be denoted by $[(f,g)]_\oplus$ when we want to emphasize the identification; however, we often just write $(f,g)$ for the sake of brevity. We can now state the uniqueness of the population potentials. We emphasize that this result holds for costs~$c$ that are merely continuous, in contrast to related results in the literature that rely on differentiating~$c$ (cf.\ \cref{se:methodology}).

 \begin{theorem}\label{th:uniqueness}
   Let $c\in\mathcal{C}(\Omega\times\Omega')$ and let $(P,Q)$ satisfy \cref{as:marginals}. Then the associated potentials $(f_*,g_*)$ are unique in~$\Banach$. That is, there exist $(f_*,g_*)\in \mathcal{C}(\Omega)\times \mathcal{C}(\Omega')$ solving~\eqref{eq:FOC.pw}, and any pair solving~\eqref{eq:FOC.pw} is of the form $(f_*+a,g_*-a)$ for some $a\in\R$.
 \end{theorem}

Next, we present the central limit theorems. Given (population) marginals $P,Q$ satisfying \cref{as:marginals}, we consider the two-sample case. Let \(X_1, \dots, X_n\) be i.i.d.\ samples from \(P\) and let \(Y_1, \dots, Y_n\) be i.i.d.\ samples from \(Q\); more precisely, the two sequences are independent and defined on the common probability space \((\boldsymbol{\Omega}, \mathcal{A}, \mathbb{P})\). The random samples give rise to the empirical measures $P_{n}=\frac{1}{n}\sum_{i=1}^{n}\delta_{X_i} $ and $Q_{n}=\frac{1}{n}\sum_{i=1}^{n}\delta_{Y_i}$, which in turn lead to empirical potentials $(f_n,g_n)$ by \cref{pr:ROTprelims}. More precisely, \cref{pr:ROTprelims} yields $(f_n,g_n)$ satisfying~\eqref{eq:FOC.as} on the supports ${\rm supp}(P_n)$ and ${\rm supp}(Q_n)$, respectively, but we can extend $(f_n,g_n)$ to continuous functions on $\Omega\times\Omega'$. While $(f_n,g_n)$ are not unique, we choose and fix them. Note that all these objects are random; that is, depend on the parameter $\omega\in\boldsymbol{\Omega}$ determining the realization of the random samples. We choose $(f_n,g_n)$ such that the dependence on $\omega$ is measurable. To simplify the notation we also introduce
\begin{equation}\label{eq:xi.def}
  \xi_*(x,y):= f_*(x)+ g_*(y)-c(x,y) \quad {\rm and}\quad \xi_n(x,y):= f_n(x)+ g_n(y)-c(x,y).
\end{equation}

\begin{assumption}[Cost]\label{as:cost}
The transport cost $c$ is in $\mathcal{C}^1(\Omega\times\Omega')$.
\end{assumption}

This assumption is in force for the three central limit theorems below, in addition to \cref{as:divergence} on the divergence and \cref{as:marginals} on the population marginals $(P,Q)$. We can now state our main result.

\begin{theorem}[CLT for potentials]\label{th:CLT.pot}
Let $\mathbf{G}_P=(\mathbf{G}_P(y))_{y\in\Omega'}$ be a centered Gaussian process in $\mathcal{C}(\Omega')$ with covariance function
\begin{multline*}
    \mathbb{E}[\mathbf{G}_P(y)\mathbf{G}_P(y')]
    \\= \mathbb{E}\left\{\left(\psi'(\xi_*(X,y))-\mathbb{E}\left[ \psi'(\xi_*(X,y)) \right]\right)\left(\psi'(\xi_*(X,y'))-\mathbb{E}\left[ \psi'(\xi_*(X,y')) \right]\right)\right\} 
\end{multline*}
for $(y,y')\in\Omega'$, where $X\sim P$. Let $\mathbf{G}_Q=(\mathbf{G}_Q(x))_{x\in\Omega}$ be a centered Gaussian process in $\mathcal{C}(\Omega)$ with analogous covariance, independent of $\mathbf{G}_P$. Then the weak limit 
 $$
\sqrt{n} \left(\begin{array}{c}
     f_n-f_*  \\
     g_n-g_* 
\end{array} \right)\xrightarrow{w} -\mathbb{L}^{-1}\left[\left( \begin{array}{c}
           \frac{\mathbf{G}_Q } {\int \psi''(\xi_*(\cdot, x))  dQ(y)}   \\
            \frac{\mathbf{G}_P}{\int \psi''(\xi_*(x, \cdot)) dP(x)}  
        \end{array}\right)\right]_{\oplus} 
$$
holds in $\Banach$, where $\mathbb{L}^{-1}\in\mathcal{L}(\Banach, \Banach)$ is the inverse of the (bounded, bijective) linear operator
\begin{align*}
    \mathbb{L}: \Banach   &\to  \Banach, \qquad
   (f,g) \mapsto \left( \begin{array}{c}
        f(\cdot) +  \frac{\int \psi''(\xi_*(\cdot,y)) g(y) dQ(y)}{\int \psi''(\xi_*(\cdot,y)) dQ(y)} \\
         g(\cdot) + \frac{\int \psi''(\xi_*(x,\cdot))  f(x) dP(x)}{\int \psi''(\xi_*(x,\cdot))  dP(x)}
   \end{array}\right).
\end{align*}
\end{theorem}

\begin{theorem}[CLT for costs]\label{th:CLT.cost}
   We have 
    $$ \sqrt{n}\big( \ROT(P_n,Q_n) -\ROT(P,Q) \big)\xrightarrow{w} N(0,\sigma^2) $$
    where the variance $\sigma^2$ is that of the random variable
    \begin{align}\label{eq:variance.CLT.cost}
	f_*(X) + g_*(Y) - \bigg(\int \psi\left(\xi_*(x,Y)\right) dP(x)
	+ \int \psi\left(\xi_*(X,y) \right) dQ(y) \bigg)
    \end{align}
    for $(X,Y)\sim P\otimes Q$.
\end{theorem}
Let $\pi\in\Pi(P,Q)$ denote the optimal coupling for the population marginals and let $\pi_n\in\Pi(P_n,Q_n)$ denote the empirical counterpart. 
To state the central limit theorem for $\pi_n\to\pi$, recall that $[(f,g)]_\oplus$ denotes the equivalence class of $(f,g)\in\mathcal{C}(\Omega)\times\mathcal{C}(\Omega')$ in $\Banach$. To streamline the notation, we also define the operator
$$ [(f,g)]_\oplus\mapsto \oplus(f,g):=f\oplus g \in \mathcal{C}(\Omega\times \Omega').$$

\begin{theorem}[CLT for couplings]\label{th:CLT.plans}
    For any bounded and measurable $\eta: \Omega\times\Omega'\to\mathbb{R}$, 
    $$ \sqrt{n}\left(\int \eta d(\pi_n - \pi) \right) \xrightarrow{w}  N(0, \sigma^2(\eta)) $$
    with $\sigma^2(\eta)=\Var(V_X + V_Y)$ where, for $Z\in\{X,Y\}$,
    \begin{align*}
      V_Z :=\mathbb{E}\Bigg[  \int  U(x,y,X,Y) \bar{\eta}(x,y) \psi'(\xi_*(x,y)) d(P\otimes Q)(x,y)
        -\bar{\eta}(X,Y)\psi' (\xi_*(X,Y))  \Bigg\vert Z\Bigg]
    \end{align*}
    for $(X,Y)\sim P\otimes Q$ and $ \bar{\eta}:= \eta - \int \eta \psi''(\xi_*)  d(P\otimes Q)$ and
    $$
     U(\cdot,\cdot,X,Y) := \oplus\left( \mathbb{L}^{-1}\left[\left( \begin{array}{c}
           \frac{ \psi' (\xi_*( \cdot ,Y))  } {\int \psi''(\xi_*( \cdot ,y))  dQ(y)}   \\
            \frac{ \psi' (\xi_*(X, \cdot )) }{\int \psi''(\xi_*(x, \cdot ))  dP(x)}  
        \end{array}\right)\right]_{\oplus} \right) .
    $$
\end{theorem}

\begin{remark}[Regularization parameter]\label{rk:general.regularization.parameter.main}
  Let us detail how the main results, stated above for $\eps=1$, read in the case of a general regularization parameter $\eps>0$. The following can be deduced by the scaling argument in \cref{rk:general.regularization.parameter}. First, $(f_*,g_*)$ are replaced by the rescaled potentials $(f_\eps,g_\eps)$ as defined in \cref{rk:general.regularization.parameter}, and $\xi_*$ is replaced by
  $$
  \xi_\eps(x,y)=\frac{f_\eps(x)+ g_\eps(y)-c(x,y)}{\eps}.
  $$
  Apart from those changes, the formula for $\mathbb{L}$ remains unchanged. In the statement of \cref{th:CLT.pot}, we further add a multiplicative factor $\eps$ in front of $\mathbb{L}^{-1}$. While in \cref{th:CLT.cost}, we replace $\ROT$ by ${\textstyle\ROT_\eps}$ on the left-hand side and add a multiplicative factor $\eps$ in front of the bracket in \eqref{eq:variance.CLT.cost}. Finally, the statement of \cref{th:CLT.plans} requires not further changes, with the understanding that $\pi,\pi_n$ are now the optimal couplings for ${\textstyle\ROT_\eps}$.
\end{remark}

\subsection{Proof strategy}\label{se:proof.strategy}

Our arguments center on the potentials $(f_*,g_*)$ and their empirical counterparts $(f_n,g_n)$. The first-order condition of optimality~\eqref{eq:FOC.pw} can be written in operator form as $\Gamma(f_*,g_*)=(1,1)$ and $\Gamma_n(f_n,g_n)=(1,1)$. In fact, for reasons discussed later, we work with closely related operators $\tilde\Gamma$ and $\tilde\Gamma_n$. Linearizing the first-order condition yields an identity of the form
\begin{equation}
    \label{eq:development1.sketch}
    \tilde\Gamma(f_n,g_n) -\tilde\Gamma_n(f_n,g_n)= \mathbb{L} (f_n-f_*,g_n-g_*) + (\dots)
\end{equation}
where $\mathbb{L}$ is a linear operator and $(\dots)$ are terms that do not cause difficulties.

The first step is to show the invertibility of $\mathbb{L}$ (\cref{pr:Invertible}). To this end, we show that $\mathbb{L}={\rm Id}+\mathbb{A}$ where $\mathbb{A}$ is a compact operator. Thus by Fredholm's alternative, \(\mathbb{L}\) is invertible if (and only if) \(-1\) is not an eigenvalue of \(\mathbb{A}\). To prove the latter, suppose that \((f, g)\) is an eigenvector corresponding to the eigenvalue \(-1\) of \(\mathbb{A}\). We first show that
\begin{equation}
    \label{IntroAlternative}
    \|(f, g)\|_\oplus = \|\mathbb{A}(f, g)\|_\oplus,
\end{equation}
which implies that \(f\) is constant on the set $ \cup_{y\in S_Q({x})}S_P({y})$ for all  $x\in \argmax |f|$, where $S_Q(x)$ and $S_P(y)$ are the sections of the set $S=\{(x,y)\in\Omega\times\Omega':\psi''(\xi_*(x,y))> 0\}$. In the case of EOT, $S$ is the full space and one immediately concludes that \(f\) is constant, completing the proof. In the general case, $S$ can be smaller (even sparse, see \cref{rk:S.and.support}) and then $\cup_{y\in S_Q({x})}S_P({y})$ does not cover the whole space $\Omega$ in general. Instead, we prove that \eqref{IntroAlternative} implies that \(f\) is constant in a neighborhood of fixed radius \(\alpha\) around the set \(\argmax |f|\). We then iterate this reasoning to show that \(f\) is constant on the whole space~$\Omega$.

Once the invertibility of $\mathbb{L}$ is shown, we focus on the central limit theorem. A high-level perspective on the proof was already given in \cref{se:methodology} based on Z-estimation. Here, we give a more pedestrian sketch. The basic idea is that if we had a central limit theorem for the left hand side of  \eqref{eq:development1.sketch}, then we could apply $\mathbb{L}^{-1}$ on both sides and deduce the desired result for  $(f_n-f_*,g_n-g_*)$. To follow this strategy, one would like to replace the left hand side $\tilde\Gamma(f_n,g_n) -\tilde\Gamma_n(f_n,g_n)$ by the more tractable expression $\tilde\Gamma(f_*, g_*) - \tilde\Gamma_n(f_*, g_*)$. Indeed, general arguments show that the latter expression satisfies a central limit theorem (\cref{lemma:CLTInHolder}). Let us proceed with that replacement, by defining the difference 
\begin{equation*}%
\Delta_n := [\tilde\Gamma - \tilde\Gamma_n](f_*, g_*) - [\tilde\Gamma  - \tilde\Gamma_n](f_n, g_n)
\end{equation*}
and hence
\begin{equation*}
    \tilde\Gamma(f_*,g_*) -\tilde\Gamma_n(f_*,g_*)= \mathbb{L} (f_n-f_*,g_n-g_*) + \Delta_n + (\dots) .
\end{equation*}
At this point we would like an estimate along the lines of $\|\Delta_n\|_{\oplus}=o\left(\|(f_*-f_n,g_*-g_n)\|_\oplus \right)$. A first attempt may be to bound $\|\Delta_n\|_{\oplus}$ in terms of
\begin{equation}\label{eq:whatOneWants.sketch}
    \|(f_* - f_n, g_* - g_n)\|_\oplus \,\sup_{\|(f,g)\|_\oplus \leq 1} \|[\tilde\Gamma_n - \tilde\Gamma](f,g)\|_\infty.
\end{equation}
However, this expression does not behave like \( o(\|(f_* - f_n, g_* - g_n)\|_\oplus) \). Indeed, the empirical measures would have to converge in total variation for the supremum over continuous functions to be of order $o(1)$. Loosely speaking, this failure arises because the unit ball of \( \Banach \) is not compact. We resolve this with a finer analysis decomposing $\Delta_n$ into two sequences, one that is valued in a compact subset of $\mathcal{C}^0$ and one whose uniform norm converges to zero almost surely (\cref{lemma:develomentIncompact}).

This leads to a key decomposition of $(f_* - f_n, g_* - g_n)$ into three sequences (\cref{pr:three.terms}): two as just described, and a third arising from the aforementioned central limit theorem in \cref{lemma:CLTInHolder}. A crucial detail is that we establish this central limit theorem not only in $\Banach$ but for a stronger norm whose unit ball is compactly embedded in the uniform topology and hence satisfies the universal Glivenko--Cantelli property, avoiding the roadblock~\eqref{eq:whatOneWants.sketch}. On the strength of the decomposition, we derive in \cref{Coro:exchangeGamman} that
  $$ \| \Delta_n\|_{\oplus}=o_{\mathbb{P}}\left(\|(f_*-f_n,g_*-g_n)\|_\oplus + n^{-\frac{1}{2}}\right).$$
Using this, the three-sequence decomposition, and the central limit theorem for two-sample $U$-statistics, \cref{se:proof.CLT.end} completes the proofs of our main results.

\section{Linearization of the Optimality Condition}\label{se:linearization}

Throughout this section, the transport cost $c\in\mathcal{C}(\Omega\times\Omega')$ is fixed and $(P,Q)$ are fixed marginals satisfying \cref{as:marginals}. (Differentiability of $c$ is not assumed.) Note that \cref{pr:ROTprelims} applies not only to those population marginals $(P,Q)$ but also to the empirical measures $(P_n,Q_n)$ of their i.i.d.\ samples. We choose continuous versions of the potentials as detailed in  \cref{pr:ROTprelims}. While at this point we do not yet know the uniqueness of the population potentials, we may choose and fix one such pair $(f_*,g_*)$. By~\eqref{eq:FOC.pw}, the potentials $(f_*,g_*)\in\mathcal{C}(\Omega)\times\mathcal{C}(\Omega')$ are a solution of the equation $\Gamma(f_*,g_*)=(1 ,1 )$, where
\begin{align*}
    \Gamma: \Banach   &\to  \mathcal{C}(\Omega)\times \mathcal{C}(\Omega'), \\
   (f,g) &\mapsto  \left( \begin{array}{c}\Gamma^{(1)}(f,g)
 \\  \Gamma^{(2)}(f,g)\end{array}\right)=\left( \begin{array}{c}
         \int \psi' (f(\cdot)+g(y)-c(\cdot,y)) dQ(y) \\
         \int \psi'(f(x)+g(\cdot)-c(x,\cdot))  dP(x)
   \end{array}\right).
\end{align*}
Similarly, the empirical potentials $(f_n,g_n)$  are a solution of $\Gamma_n(f_n,g_n)=(1,1)$, where
   \begin{align*}
    \Gamma_n: \Banach   &\to  \mathcal{C}(\Omega)\times \mathcal{C}(\Omega'), \\
   (f,g) &\mapsto \left( \begin{array}{c}\Gamma^{(1)}_n(f,g)
 \\  \Gamma^{(2)}_n(f,g)\end{array}\right)=\left( \begin{array}{c}
   \int \psi'(f(\cdot)+g(y)-c(\cdot,y)) dQ_n(y) \\
        \int \psi'(f(x)+g(\cdot)-c(x,\cdot)) dP_n(x)
   \end{array}\right).
\end{align*}
As mentioned in the preceding section, we choose and fix empirical potentials $(f_n,g_n)$ that are continuous on $\Omega\times\Omega'$ and depend measurably on $\omega\in\boldsymbol{\Omega}$.

Recall the shorthand $\xi_*(x,y)= f_*(x)+ g_*(y)-c(x,y)$ from~\eqref{eq:xi.def}. The following lemma details the first-order development of the operator $$\tilde\Gamma:\Banach\to\mathcal{C}(\Omega)\times \mathcal{C}(\Omega'), \qquad (f,g)\mapsto\tilde\Gamma(f,g) := \left( \begin{array}{c}\frac{\Gamma^{(1)}}{\int \psi''(\xi_*(\cdot,y))  dQ(y)}
 \\  \frac{\Gamma^{(2)}}{\int \psi''(\xi_*(x,\cdot))dP(x)}\end{array}\right).
 $$ 
Here the denominators are bounded away from zero by \cref{pr:ROTprelims}. We introduce those denominators so that the derivative $\mathbb{L}$ (see~\eqref{eq:L} below) becomes an operator $\Banach\to\Banach$ that is a perturbation of the identity. (A priori, these operators could depend on the chosen and fixed potentials $(f_*,g_*)$, which causes no harm. Once the proof of uniqueness is complete, it will be clear that there was in fact no ambiguity.)

\begin{lemma}\label{Lemma:Frechet}
    The operator $\tilde\Gamma:\Banach\to\mathcal{C}(\Omega)\times \mathcal{C}(\Omega')$ is continuously Fr\'{e}chet differentiable; we denote its derivative at the point $(f, g)\in\Banach$ by $\mathbb{L}_{(f, g)}$. That is, we have
    $$ \lim_{{\|(u,v)\|_{\oplus}}\to 0}\frac{ \| \tilde\Gamma(f+u,g+v)-\tilde\Gamma(f,g)-\mathbb{L}_{(f,g)}(u,v)\|_{\mathcal{C}(\Omega)\times \mathcal{C}(\Omega')} }{\|(u,v)\|_{\oplus}} =0 $$
    and the function $\Banach\ni (f,g) \mapsto \mathbb{L}_{(f,g)} \in \mathcal{L}(\Banach,\mathcal{C}(\Omega)\times \mathcal{C}(\Omega')) $ is continuous. The derivative at $(f_*,g_*)$ is given by 
        \begin{align*}
    	\mathbb{L}_{(f_*,g_*)}: \Banach   &\to  \mathcal{C}(\Omega)\times \mathcal{C}(\Omega'), \qquad
    	(u,v) \mapsto \left( \begin{array}{c}
    		u(\cdot) +  \frac{\int \psi''(f_*(\cdot)+g_*(y)-c(\cdot,y)) v(y) dQ(y)}{\int \psi''(\xi_*(\cdot,y)) dQ(y)} \\
    		v(\cdot) + \frac{\int \psi''((f_*(x)+g_*(\cdot)-c(x,\cdot))  u(x) dP(x)}{\int \psi''(\xi_*(x,\cdot))  dP(x)}
    	\end{array}\right).
    \end{align*} 
\end{lemma}

\begin{proof}
  Using that $\psi'$ is $\mathcal{C}^{1}$, the proof is a direct calculation.
\end{proof}

Composing with the quotient map, we define the operator $\mathbb{L}\in\mathcal{L}(\Banach,\Banach)$ by
\begin{align}\label{eq:L}
\mathbb{L}: (u,v)\mapsto  \left[\mathbb{L}_{(f_*,g_*)}(u,v)\right]_\oplus.
\end{align}
The main goal of this section is to show that $\mathbb{L}$ is invertible. We first state an auxiliary result on the sections of the set $S:=\{(x,y)\in\Omega\times\Omega':\psi''(\xi_*(x,y))> 0\}$ whose interpretation is discussed in \cref{rk:S.and.support} below. The proof is reported in \cref{se:proofsLinearization}.

\begin{lemma}\label{Lemma:aBallInside}
Consider the sections of  $S=\{(x,y)\in\Omega\times\Omega':\psi''(\xi_*(x,y))> 0\}$,
    $$
    S_Q(x) := \{y\in \Omega':\  \psi''(\xi_*(x,y))> 0\}, \qquad 
    S_P(y) := \{x\in \Omega:\  \psi''(\xi_*(x,y))> 0 \}.
    $$
    There exists  $\alpha>0$ such that
    $$ (x+\alpha\, \overline{\mathbb{B}})\cap \Omega\subset \cup_{y\in S_Q({x})}S_P({y}) \quad \text{for all } x\in \Omega.$$
\end{lemma}

\begin{remark}\label{rk:S.and.support}
The set $S=\{(x,y)\in\Omega\times\Omega':\psi''(\xi_*(x,y))> 0\}$ is closely related to the support of the optimal coupling $\pi$. Indeed, $\psi'\circ \xi_*$ is the density of $\pi$ by \cref{pr:ROTprelims}. For the KL and polynomial divergences detailed in \cref{ex:divergences}, $\{t:\psi'(t)>0\}=\{t:\psi''(t)>0\}$, hence $S$ is exactly the set where the density is positive, and its closure is the support of $\pi$. For the KL divergence of EOT, $\psi''>0$ on all of~$\mathbb{R}$ and thus $S=\Omega\times\Omega'$, so that \cref{Lemma:aBallInside} is trivial. For the polynomial divergences, however, it is known that the support is often sparse (see \cref{se:literature}), and then \cref{Lemma:aBallInside} is relevant.
\end{remark}

The formula in \cref{Lemma:Frechet} suggests to see $\mathbb{L}={\rm Id}+\mathbb{A}$ as a perturbation of the identity ${\rm Id}\in\mathcal{L}(\Banach,\Banach)$. To prove the invertibility of $\mathbb{L}$, we will show that the operator \begin{align}\label{eq:operatorA}
    \mathbb{A}=(\mathbb{A}_1,\mathbb{A}_2): \Banach   \to  \Banach, \qquad 
   (f,g) \mapsto\left( \begin{array}{c}
          \frac{\int \psi''(\xi_*(\cdot,y))  g(y) dQ(y)}{\int \psi''(\xi_*(\cdot,y))  dQ(y)} \\
    \frac{\int \psi''(\xi_*(x,\cdot)) f(x) dP(x)}{\int \psi''(\xi_*(x,\cdot))  dP(x)}
   \end{array}\right)
\end{align}
is compact and $-1$ is not an eigenvalue. Then, we conclude by the Fredholm alternative. While similar results are known for EOT~\cite{Carlier.Laborde.SIMA.2020,GonzalezSanz.2024.weaklimits}, we require a substantially different proof because our optimal coupling need not have full support.

\begin{proposition}\label{pr:Invertible}
The operator $\mathbb{L}\in\mathcal{L}(\Banach,\Banach)$ admits an inverse $\mathbb{L}^{-1}\in\mathcal{L}(\Banach,\Banach)$.
\end{proposition}

\begin{proof}

We first show that the operator $\mathbb{A}$ of \eqref{eq:operatorA} is compact. We  need to show that if $\{(u_n, v_n)\}_{n\in \mathbb{N}}$ is bounded in $\Banach$ then 
$\{(\mathbb{A}_1u_n, \mathbb{A}_2 v_n)\}_{n\in \mathbb{N}}$ is relatively compact in $\Banach$. Since $\{(u_n, v_n)\}$ is bounded in $\Banach$, there exists a sequence $\{a_n\}$ of real numbers such that 
 $ \|u_n-a_n\|_{\infty}  $ and 
 $ \|v_n+a_n\|_{\infty}  $ are bounded. A sufficient condition for the  relative compactness of  $\{(\mathbb{A}_1 v_n, \mathbb{A}_2 u_n)\}$ in $\Banach$ is that  $\{\mathbb{A}_1(v_n+a_n)\}$ and $\{\mathbb{A}_2(u_n-a_n)\}$ are relatively  compact in $\mathcal{C}(\Omega)$
and  $\mathcal{C}(\Omega')$, respectively. We only prove that $\{\mathbb{A}_2(u_n-a_n)\}$ is  relatively  compact, the former is analogous. By the Arzel\`{a}--Ascoli theorem, it suffices to show equicontinuity. As the function
$ \left(\int \psi''(\xi_*(x,\cdot))  dP(x)\right)^{-1} $
is continuous by \cref{pr:ROTprelims}, and since the product of an equicontinuous sequence with a continuous function (on the compact space $\Omega'$) remains equicontinuous, it moreover suffices to show that 
\begin{equation}\label{eq:A.compact.prof.equi}
  \int \psi''(\xi_*(x,\cdot)) (u_n(x)-a_n ) dP(x), \quad n\in \mathbb{N} 
\end{equation}
is equicontinuous. Indeed, let $\rho$ be a modulus of  continuity for $\int \psi''(\xi_*(x,\cdot)) dP(x)$ and $C=\sup_n \|u_n - a_n\|_\infty$, then $C\rho$ is a modulus of continuity for the functions in~\eqref{eq:A.compact.prof.equi}. This completes the proof that $\mathbb{A}$ is compact.

Next, we show that $-1$ is not an eigenvalue of $\mathbb{A}$. Suppose that
\begin{equation}\label{eq:lem45}
    f + a=-\mathbb{A}_1(g) \quad {\rm and}\quad 
        g - a=-\mathbb{A}_2(f)  
\end{equation}
for some  $(f,g)\in \mathcal{C}(\Omega)\times\mathcal{C}(\Omega')$ and $a \in \mathbb{R}$; we need to prove that $[(f, g)]_{\oplus} = [(0, 0)]_{\oplus}$. As a first step, we show that $a = 0$ must hold in \eqref{eq:lem45}. Using \eqref{eq:lem45} together with the definition of~$\mathbb{A}_2$,  we have
\begin{equation}\label{eq:ToCancel-c-1}
	{\int \psi''(\xi_*(x,y)) (f(x)+  g(y)) dQ(y)} =-a{\int \psi''(\xi_*(x,y))  dQ(y)}, \quad \text{for all }x\in \Omega
\end{equation}
and similarly by the definition of $\mathbb{A}_1$,
\begin{equation}\label{eq:ToCancel-c-2}
	{\int \psi''(\xi_*(x,y)) (f(x)+  g(y)) dP(x)} =a{\int \psi''(\xi_*(x,y))  dP(x)}, \quad \text{for all }y\in \Omega'.
\end{equation}
Integrating in \eqref{eq:ToCancel-c-1}  and \eqref{eq:ToCancel-c-2} w.r.t.\ $P$ and $Q$, respectively, we get  
$$ a{\int \psi''(\xi_*(x,y))  d(P\otimes Q)(y)}=0, $$
which implies $a=0$. Hence, \eqref{eq:lem45} simplifies to 
$f=-\mathbb{A}_1(g)$ and $g=-\mathbb{A}_2(f), 
$ which we can concatenate to deduce
\begin{equation}\label{eq:concatA1A2}
  f=\mathbb{A}_1 \mathbb{A}_2(f).
\end{equation}

Next, we claim that for any $x\in \argmax |f|$, $f$ is constant on the set $\cup_{y\in S_Q({x})}S_P({y})\subset\Omega$. %
Indeed, taking norms on both sides of ~\eqref{eq:concatA1A2} yields 
$$ \|f\|_{\infty}=\| \mathbb{A}_1 \mathbb{A}_2(f) \|_{\infty}.$$
It then suffices to show that, for any $h\in\mathcal{C}(\Omega)$, 
$$ \| \mathbb{A}_1 \mathbb{A}_2(h) \|_{\infty} < \|h\|_{\infty} $$
 unless $h$ is constant on the set $\cup_{y\in S_Q({x})}S_P({y}) $ for all $x\in \argmax |h|$.  To see the latter, fix $x_*\in \argmax |h|$ and define the probability measure 
$$ A\mapsto \mu_{x_*}(A)= \frac{\int \psi''(\xi_*(x_*,y))  \frac{\int_A \psi''(\xi_*(x,y))  dP(x)}{\int \psi''(\xi_*(x,y))  dP(x)} dQ(y)}{\int \psi''(\xi_*(x_*,y))  dQ(y)}$$
on $\Omega$. Then $\int h(x) d\mu_{x_*}(x)=\mathbb{A}_1 \mathbb{A}_2(h)(x_*)$ for any $h\in\mathcal{C}(\Omega)$. In particular, using also~\eqref{eq:concatA1A2},  $$\|f\|_\infty= |f(x_*)| =|\mathbb{A}_1 \mathbb{A}_2(f)(x_*)|= \left| \int f(x) d\mu_{x_*}(x)  \right|.$$
As $f$ is continuous, this implies that  $f(x)=\|f\|_\infty$ for all $x\in {\rm supp}(\mu_{x_*})$, the support  of $\mu_{x_*}$. Observing that ${\rm supp}(\mu_{x_*})$ contains the set $\bigcup_{y\in S_Q({x_*})}S_P({y})$ completes the proof of the claim that $f$ is constant on the set $\cup_{y\in S_Q({x})}S_P({y}) $ for all $x\in \argmax |f|$.

Next, we show that $f = \|f\|_\infty$ on $\Omega$. By  \cref{Lemma:aBallInside} there is $\alpha>0$ such that $$ (x+\alpha \overline{\mathbb{B}})\cap \Omega\subset \cup_{y\in S_Q({x})}S_P({y}) \quad \text{for all }  x\in \Omega.$$ Fix $x_*\in \argmax |f|$. The above then implies $f=\|f\|_\infty$ on  $(x_*+\alpha \overline{\mathbb{B}})\cap \Omega$. We iterate this argument to prove that $f = \|f\|_\infty$ on $\Omega$. Indeed, let $\bar{x} \in \Omega$ be arbitrary; we show that $f(\bar{x}) = \|f\|_\infty$. As $\Omega$ is compact and connected, there exist $N \in \mathbb{N}$ and $x_1, x_2, \dots, x_N \in \Omega$ with $x_1 = x_*$, $x_N = \bar{x}$ and $|x_i - x_{i+1}| \leq \alpha/2$ (e.g., such $x_i$ arise from an open cover of $\Omega$ by balls of radius $\alpha/2$). We know $f(x_1) = \|f\|_\infty$ and the above argument shows that if $f(x_i)=\|f\|_\infty$ then also $f = \|f\|_\infty$ on $(x_i+\alpha \overline{\mathbb{B}})\cap \Omega$, which in particular yields $f(x_{i+1}) = \|f\|_\infty$. Inductively, we obtain $f(\bar{x}) = f(x_N) = f(x_1)=\|f\|_\infty$. This completes the proof that $f = \|f\|_\infty$ on $\Omega$. 

Using the fact that $g=-\mathbb{A}_2(f)$, we derive that $\|g\|_\infty=-\|f\|_\infty$, implying  that $(f, g) \sim_{\oplus} (0, 0)$. In summary, we have shown that $-1$ is not an eigenvalue of $\mathbb{A}$, and hence that the linear operator $\mathbb{L}={\rm Id} + \mathbb{A}$ is injective. By Fredholm's alternative, it follows that $\mathbb{L}:\Banach\to\Banach$ is also surjective. As $\Banach$ is a Banach space, the inverse is necessarily continuous, completing the proof. 
\end{proof}

As a by-product of the invertibility of $\mathbb{L}$, we obtain the uniqueness of the potentials.

 \begin{proof}[Proof of \cref{th:uniqueness}]
 Let $(f_*,g_*)$ be as above and let $(\tilde{f}_*,\tilde{g}_*)$ be another solution of the first-order condition~\eqref{eq:FOC.pw}. We first note that any convex combination 
 $(\lambda \tilde{f}_*+(1-\lambda) f_*,\lambda \tilde{g}_*+(1-\lambda) g_*)$, for $\lambda\in (0,1)$, is again a solution of \eqref{eq:FOC.pw}. Indeed, by \cref{pr:ROTprelims}, the set of solutions of \eqref{eq:FOC.pw} is the set of (continuous) optimizers of the dual problem \eqref{eq:dual1}, and the latter set is convex since \eqref{eq:dual1} is a concave maximization problem. We can now use \cref{Lemma:Frechet} to derive
 $$ 0= \frac{d}{d\lambda}\bigg\vert_{\lambda=0}\tilde{\Gamma}(\lambda \tilde{f}_*+(1-\lambda) f_*,\lambda \tilde{g}_*+(1-\lambda) g_*)= \mathbb{L}( \tilde{f}_*- f_*, \tilde{g}_*- g_*).$$
As $\mathbb{L}$ is invertible by \cref{pr:Invertible}, it follows that $( \tilde{f}_*- f_*, \tilde{g}_*- g_*)=0$ in $\Banach$.
 \end{proof}
 
\section{Proofs of the Central Limit Theorems}\label{se:proofs.of.CLTs}

 \subsection{Consistency of the Potentials}\label{se:consistency}
 
 We first state the consistency of the empirical potentials towards the population counterpart. 

\begin{lemma}\label{lemma.A.s}   
    For each $n\in \mathbb{N}$, let $(f_n,g_n)\in \Banach$ be any solution of $\Gamma_n(f_n,g_n)=(1,1)$. Then 
    $$ \|(f_n,g_n)-(f_*,g_*)\|_{\oplus}\xrightarrow{\text{a.s.}} 0 .$$
\end{lemma}

\begin{proof}%
Recall that empirical quantities such as $P_{n},f_{n},\dots$ implicitly depend on the realization 
$X_1(\omega), \dots, X_n(\omega), Y_1(\omega), \dots, Y_n(\omega)$ of the sample. 
To make this dependence explicit, we add a superscript $\omega$. As $X_1, \dots, X_n, Y_1, \dots, Y_n$ are i.i.d., the Glivenko--Cantelli theorem implies that there is a set $\boldsymbol{\Omega}_1\subset\boldsymbol{\Omega}$ with $\mathbb{P}(\boldsymbol{\Omega}_1)=1$ such that for all $\omega\in\boldsymbol{\Omega}_1$,
$$ \sup_{\|f\|_{0,1}\leq 1}\left\vert \int f d(P^{\omega}_n-P) \right\vert \to0\quad {\rm and} \quad \sup_{\|g\|_{0,1}\leq 1}\left\vert \int g d(Q^{\omega}_n-Q) \right\vert \to 0.$$ 
By~\cref{pr:ROTprelims} there is a constant $K>0$ such that $\|(f_n^{\omega}\oplus g_n^{\omega})\|_{0,1}\leq K$ for all $n\geq1$ and $\omega\in\boldsymbol{\Omega}$. Recalling the notation $\xi_{n}(x,y):=f_n(x)+ g_n(y)-c(x,y)$, it follows that $\|\psi'(\xi_{n}^\omega)\|_{0,1}\leq K'$ for a constant~$K'$. We conclude that for all $(x_{0},y_{0})\in\Omega\times\Omega'$ and all $\omega\in\boldsymbol{\Omega}_1$,
\begin{align}\label{eq:rightHandSideZero}
  (1,1)- \Gamma(f_{n}^\omega,  g_{n}^\omega)(x_0,y_0) 
  & = \big[\Gamma_{n}^{\omega}(f_{n}^\omega,  g_{n}^\omega)-\Gamma(f_{n}^\omega,  g_{n}^\omega)\big](x_0,y_0)  \nonumber\\
  &= \left( \begin{array}{c}
   \int \psi'(\xi_{n}^\omega(x_0,y)) d(Q^\omega_n-Q)(y) \\
        \int \psi'(\xi_{n}^\omega(x,y_0)) d(P^\omega_n-P)(x)
   \end{array}\right) \to 0.
 \end{align}
Fix $\omega\in\boldsymbol{\Omega}_1$. After passing to a subsequence, $(f_{n}^\omega,  g_{n}^\omega)$ converges in $\Banach$ to a limit $(f_{0}^\omega,  g_{0}^\omega)$. We have $\Gamma(f_{n}^\omega,  g_{n}^\omega)\to \Gamma(f_{0}^\omega,  g_{0}^\omega)$ by the continuity of $\Gamma$, and thus~\eqref{eq:rightHandSideZero} shows that $\Gamma(f_{0}^\omega,  g_{0}^\omega)=(1,1)$. The uniqueness in \cref{th:uniqueness} now implies that $(f_{0}^\omega,  g_{0}^\omega)=(f_{*},g_{*})$ as elements of~$\Banach$. More precisely, this argument shows that any convergent subsequence converges to $(f_{*},g_{*})$, and hence that the whole sequence $(f_{n}^\omega,  g_{n}^\omega)$ converges to $(f_{*},g_{*})$ in $\Banach$. As $\mathbb{P}(\boldsymbol{\Omega}_1)=1$, this amounts to the claimed a.s.\ convergence.
\end{proof}%
\subsection{Key Decomposition}\label{se:key.decomposition}

To facilitate the reading we will assume from now on that $(f_n,g_n)$ are appropriately shifted so that
$$ \|(f_n-f_*,g_n-g_*)\|_{\oplus}=\|f_n-f_*\|_{\infty}+\|g_n-g_*\|_{\infty}.  $$
We also fix a H\"{o}lder exponent $\beta\in (0,1)$. 
As \( \int \psi''(\xi_*(\cdot, y)) dQ(y) \) is bounded away from zero and continuous (\cref{pr:ROTprelims}), the linear operator 
\[
\mathbb{F}: \mathcal{C}(\Omega) \times \mathcal{C}(\Omega') \ni (f,g) \mapsto \left( \begin{array}{c} 
\frac{f}{\int \psi''(\xi_*(\cdot,y)) dQ(y)} \\
\frac{g}{\int \psi''(\xi_*(x,\cdot)) dP(x)} 
\end{array} \right) \in \mathcal{C}(\Omega) \times \mathcal{C}(\Omega')
\]
is bounded and bijective. We define the Banach space \( \mathcal{B}_{\mathbb{F},\beta}\) of all functions %
\( (f,g)\in \mathcal{C}(\Omega) \times \mathcal{C}(\Omega') \) with finite norm  
\[
\| (f, g) \|_{\mathbb{F},\beta} :=  \| \mathbb{F}^{-1}(f, g) \|_{0, \beta}.
\]
This norm is stronger than the uniform norm. Indeed, a crucial fact for our subsequent arguments is that the unit ball of \( \mathcal{B}_{\mathbb{F},\beta} \) is compactly embedded in \( \mathcal{C}(\Omega) \times \mathcal{C}(\Omega') \) (by the same reasoning as in \cite[Lemma~6.33]{GilbargTrudinger.83}). 

It follows from \cref{th:CLT.Holder} of \cref{se:HolderCLT} that the central limit theorem
\[
\sqrt{n} \left(\Gamma(f_*, g_*) - \Gamma_n(f_*, g_*) \right)
=\sqrt{n}\left( \begin{array}{c}
         \int \psi' (\xi_*(\cdot,y)) d(Q-Q_n)(y) \\
         \int \psi'(\xi_*(x,\cdot))  d(P-P_n)(x)
   \end{array}\right)
\xrightarrow{w} \left( \begin{array}{c}
{\mathbf{G}_Q} \\
{\mathbf{G}_P}
\end{array} \right)
\]
holds in \( \mathcal{C}^{0, \beta}(\Omega) \times \mathcal{C}^{0, \beta}(\Omega') \), where $\mathbf{G}_Q, \mathbf{G}_P$ are Gaussian processes as detailed in \cref{th:CLT.pot}. As a consequence, the isometric image \( \tilde\Gamma_n = \mathbb{F}(\Gamma_n) \) satisfies the following central limit theorem in the space~\( \mathcal{B}_{\mathbb{F},\beta} \).

\begin{lemma}\label{lemma:CLTInHolder}
The following weak limit holds in $\mathcal{B}_{\mathbb{F},\beta}$,
$$ \sqrt{n}\,[\tilde\Gamma-\tilde\Gamma_n](f_*,g_*):= \sqrt{n} \left(\tilde\Gamma(f_*,g_*) -\tilde\Gamma_n(f_*,g_*)\right) \xrightarrow{w}\left( \begin{array}{c}
           \frac{\mathbf{G}_Q } {\int \psi''(\xi_*( \cdot ,y)) dQ(y)}   \\
            \frac{\mathbf{G}_P}{\int \psi''(\xi_*(x, \cdot ))  dP(x)}  
        \end{array}\right).$$
In particular, the sequence $\sqrt{n}\,[\tilde\Gamma-\tilde\Gamma_n](f_*,g_*)$ is tight in $\mathcal{B}_{\mathbb{F},\beta}$.
\end{lemma}

The goal of this subsection, and indeed the key technical result, is the following decomposition of $(f_n - f_*,g_n - g_*)$ into three terms.

\begin{proposition}\label{pr:three.terms}
    There exist random functions   $ U_n, V_n, W_n\in \mathcal{C}(\Omega)$ and $ U_n', V_n', W_n'\in \mathcal{C}(\Omega')$, and compact sets  $\mathcal{K}\subset\mathcal{C}(\Omega)$ and $\mathcal{K}'\subset\mathcal{C}(\Omega')$, such that
    \begin{enumerate}
    \item $(U_n, U_n')$ takes values in $\mathcal{K} \times \mathcal{K}'$ for all $n$,
    \item    $ \|(V_n,V_n')\|_\infty\xrightarrow{\text{a.s.}} 0$,
    \item $\mathbb{L}(W_n,W_n')$ is tight in $\mathcal{B}_{\mathbb{F},\beta}$,    
    \item for all $n$,  
    \begin{equation}
    \label{eq:three.terms}
    \left(\begin{array}{c}
      f_n - f_*     \\
      g_n - g_*  
    \end{array} \right) 
    = \| f_n - f_*, g_n - g_*\|_\oplus 
    \left(\begin{array}{c}
       U_n + V_n     \\
       U_n' + V_n'  
    \end{array} \right) 
    + n^{-\frac{1}{2}}
    \left(\begin{array}{c}
       W_n     \\
       W_n'  
    \end{array} \right).
\end{equation}      
\end{enumerate}  
\end{proposition}

Next, we collect the higher-level steps of the proof, while outsourcing the more technical parts to the subsequent lemmas.

\begin{proof}[Proof of \cref{pr:three.terms}]
By \cref{Lemma:Frechet} there are random functions \((V_{n,\mathbb{L}}, V_{n,\mathbb{L}}') \in \mathcal{C}(\Omega) \times \mathcal{C}(\Omega')\) with \(\|(V_{n,\mathbb{L}},V_{n,\mathbb{L}}')\|_\infty \xrightarrow{\text{a.s.}} 0\) such that
\[
\tilde\Gamma(f_n, g_n) - \tilde\Gamma(f_*, g_*) = \mathbb{L} (f_n - f_*, g_n - g_*) - \|(f_n - f_*, g_n - g_*)\|_{\oplus} (V_{n,\mathbb{L}},V_{n,\mathbb{L}}').
\]
Using the fact that \(\Gamma(f_*, g_*) = (1,1) = \Gamma_n(f_n, g_n)\) and setting 
$$\tilde\Gamma_n : = \left( \begin{array}{c}\frac{\Gamma_n^{(1)}}{\int \psi''(\xi_*(\cdot,y)) dQ(y)}
 \\  \frac{\Gamma^{(2)}_n}{\int \psi''(\xi_*(x,\cdot))dP(x)}\end{array}\right),$$
this yields
    \begin{equation}
        \label{eq:development1}
        \tilde\Gamma(f_n,g_n) -\tilde\Gamma_n(f_n,g_n)= \mathbb{L} (f_n-f_*,g_n-g_*) - \|(f_n-f_*,g_n-g_*)\|_{\oplus} (V_{n,\mathbb{L}},V_{n,\mathbb{L}}').
    \end{equation}
Recall from \cref{pr:Invertible} that \( \mathbb{L} \) is a bounded bijection. Moreover, $\|\tilde\Gamma(f_* - f_n, g_* - g_n)\|_{\Banach}\to0$ a.s.\ by \cref{lemma.A.s} and the continuity of $\tilde\Gamma$. Setting also 
\begin{equation}\label{eq:Delta.n}
\Delta_n := [\tilde\Gamma - \tilde\Gamma_n](f_*, g_*) - [\tilde\Gamma  - \tilde\Gamma_n](f_n, g_n),
\end{equation}
we deduce
\begin{multline*}
    (f_n - f_*, g_n - g_*)\\
    =\mathbb{L}^{-1}([\tilde\Gamma - \tilde\Gamma_n] (f_*, g_*))-\mathbb{L}^{-1}(\Delta_n)+ \|(f_n - f_*, g_n - g_*)\|_{\oplus} \mathbb{L}^{-1}(V_{n,\mathbb{L}}, V'_{n,\mathbb{L}}).
\end{multline*}
Set $(W_n,W_n')=n^{\frac{1}{2}}\mathbb{L}^{-1}([\tilde\Gamma - \tilde\Gamma_n] (f_*, g_*))$. Then $\mathbb{L}(W_n,W_n')$ is tight by \cref{lemma:CLTInHolder}  and $\mathbb{L}^{-1}([\tilde\Gamma - \tilde\Gamma_n] (f_*, g_*))=n^{-\frac{1}{2}}(W_n,W_n')$.
Therefore,
\begin{align*}
    (f_n - f_*, g_n - g_*) 
    = n^{-\frac{1}{2}}(W_n,W_n') -\mathbb{L}^{-1}(\Delta_n) +\|(f_n - f_*, g_n - g_*)\|_{\oplus} \mathbb{L}^{-1}(V_{n,\mathbb{L}}, V'_{n,\mathbb{L}}).
\end{align*}
\Cref{lemma:develomentIncompact} below derives a decomposition for the remaining term $\Delta_n$ which completes the proof of \cref{pr:three.terms} after recalling that $\mathbb{L}^{-1}$ is continuous. Specifically, the function $(V_n,V_n')$ in the assertion of \cref{pr:three.terms} arises from combining the function $\mathbb{L}^{-1}(V_{n,\mathbb{L}}, V'_{n,\mathbb{L}})$ above with the function $\mathbb{L}^{-1}(\tilde V_n, \tilde V'_n)$ from \cref{lemma:develomentIncompact}.
\end{proof}

The following lemma encapsulates the main step of the preceding proof of \cref{pr:three.terms} and uses the notation $\Delta_n$ from~\eqref{eq:Delta.n}. %
 
\begin{lemma}\label{lemma:develomentIncompact}
There exist compact sets  $\tilde{\mathcal{K}}\subset\mathcal{C}(\Omega)$ and $\tilde{\mathcal{K}}'\subset\mathcal{C}(\Omega')$, and random functions   $ \tilde U_n, \tilde V_n\in \mathcal{C}(\Omega)$ and $ \tilde U_n',  \tilde V_n'\in \mathcal{C}(\Omega')$, such that
\begin{enumerate}
    \item $(\tilde U_n, \tilde U_n')$ takes values in $\tilde{\mathcal{K}} \times \tilde{\mathcal{K}}'$ for all $n$,
    \item    $ \|\tilde V_n,\tilde V_n'\|_\infty\xrightarrow{\text{a.s.}} 0$,
    \item for all $n$,  
    $$ \Delta_n= \left(\begin{array}{c}
      \|f_n-f_*\|_\infty \tilde V_n     \\
       \|g_n-g_*\|_\infty \tilde V_n'  
    \end{array} \right)+ \left(\begin{array}{c}
      \|g_n-g_*\|_\infty \tilde U_n     \\
       \|f_n-f_*\|_\infty \tilde U_n'  
    \end{array} \right).
$$
    \end{enumerate}
    
\end{lemma}

\begin{proof}%
    We prove the claims concerning the functions on~$\Omega$; the functions on~$\Omega'$ are obtained analogously.  
We start by recalling that 
$$  \Gamma^{(1)}(f_n,g_n) -\Gamma_n^{(1)}(f_n,g_n)=\int \psi'(\xi_n(\cdot, y)) d(Q-Q_n)(y)$$
and 
$$  \Gamma^{(1)}(f_*,g_*) -\Gamma_n^{(1)}(f_*,g_*)=\int \psi'(\xi_*(\cdot, y)) d(Q-Q_n)(y).$$
Set 
$$ \Delta_n^{(1)} := \Gamma^{(1)}(f_*, g_*) - \Gamma^{(1)}_n(f_*, g_*) - \big(\Gamma^{(1)}(f_n, g_n) - \Gamma^{(1)}_n(f_n, g_n)\big). $$
Note that writing $\Delta_n^{(1)}$ is a slight abuse of notation: defining analogously $\Delta_n^{(2)}$, we have
\begin{equation}\label{eq:DeltaSplit}
  \Delta_n = \left(\frac{\Delta_n^{(1)}}{\int \psi''(\xi_*(\cdot, y)) \, dQ(y)},\frac{\Delta_n^{(2)}}{\int \psi''(\xi_*(x,\cdot)) \, dP(x)} \right).
\end{equation}
The fundamental theorem of calculus implies that 
\begin{align*}
   \Delta_n^{(1)}  &=\int \{\psi'(\xi_*(\cdot, y))-\psi'(\xi_n(\cdot, y))\} d(Q-Q_n)(y)\\
    &= \int \int_0^1 \psi'' \big( (1-\lambda)\xi_n(\cdot, y) +\lambda \xi_*(\cdot, y)\big)( \xi_*(\cdot, y) - \xi_n(\cdot, y)) d\lambda d(Q-Q_n)(y).
\end{align*}
We set $\psi_n(\cdot,y) := \int_0^1 \psi''\big( (1-\lambda)\xi_n(\cdot, y) +\lambda \xi_*(\cdot, y)\big)d\lambda $ and use the definitions of $\xi_*,\xi_n$ to deduce
\begin{align}
   \Delta_n^{(1)}&=\int \psi_n(\cdot,y)(  \xi_*(\cdot, y) - \xi_n(\cdot, y))  d(Q-Q_n)(y) \nonumber\\
    &=  (f_*-f_n)\int \psi_n(\cdot,y) d(Q-Q_n)(y)
+ \int \psi_n(\cdot,y)( g_*(y)-g_n(y))  d(Q-Q_n)(y). \label{DifferenceGammas}
\end{align}
Considering the first term in~\eqref{DifferenceGammas} and adding back the denominator from~\eqref{eq:DeltaSplit}, set
$$
  \tilde V_n := \frac{1}{\int \psi''(\xi_*(\cdot, y)) \, dQ(y)}\int \psi_n(\cdot,y) d(Q-Q_n)(y).
$$
Then $\tilde V_n$ is a random variable with values in $\mathcal{C}(\Omega)$. As $\xi_n,\xi_*$ have uniformly bounded Lipschitz norms by \cref{pr:ROTprelims}, it follows from \cref{lemma:compactclass} below that $ \|\tilde V_n\|_\infty\to0$ a.s. Regarding the second term in~\eqref{DifferenceGammas}, set (using the convention $0/0:=0$ if necessary)
$$
  \tilde U_n := \frac{1}{\int \psi''(\xi_*(x,\cdot)) \, dP(x)}  \int \psi_n(\cdot,y)\frac{g_*(y)-g_n(y)}{\|g_n-g_*\|_\infty}  d(Q-Q_n)(y).
$$
Then \cref{lemma:compactclass2} below, applied with $d\mu(y):=\frac{g_*(y)-g_n(y)}{\|g_n-g_*\|_\infty}  d(Q-Q_n)(y)$, shows that $\tilde U_n$ is a random variable with values in a compact set $\tilde{\mathcal{K}}\subset\mathcal{C}(\Omega)$. %
\end{proof}

The next two technical lemmas were used in the preceding proof of \cref{lemma:develomentIncompact}.

\begin{lemma}\label{lemma:compactclass}
    For any $C\in\mathbb{\R}$, the class of functions
$$ \mathcal{F}:= \left\{\Omega' \ni y\mapsto \int_{0}^1\psi''(  h(x,y, \lambda) ) d\lambda: x\in \Omega,\  \|h\|_{0,1}\leq C\right\} $$
is relatively compact in $\mathcal{C}(\Omega)$. As a consequence, 
$$  \sup_{f\in \mathcal{F}} \left|\int f d(Q_n-Q)\right| \xrightarrow{\text{a.s.}} 0.$$
\end{lemma}
\begin{proof}
Let $\rho$ be a modulus of continuity of $\psi''$ on $[-C,C]$, i.e., 
$\psi''(t_1)-\psi''(t_2)\leq \rho(|t_1-t_2|)$ for all $t_1, t_2\in [-C,C]$, where $\rho$ is monotone increasing and $\rho(s)\to 0$ as $s\to 0$. 
    Then  the modulus of continuity of $\int_{0}^1\psi''(  h(x,\cdot, \lambda) ) d\lambda$ is bounded as follows, 
    \begin{align*}
    \Bigg\vert \int_{0}^1\psi''(  h(x,y_1, \lambda) )& d\lambda-\int_{0}^1\psi''(  h(x,y_2, \lambda) )d\lambda   \Bigg\vert \\
    &\leq \int_{0}^1 \vert \psi''(  h(x,y_1, \lambda) ) -\psi''(  h(x,y_2, \lambda) ) \vert d\lambda \\
    &\leq \int_{0}^1 \rho(|h(x,y_1, \lambda) -h(x,y_2, \lambda)| ) d\lambda\\
    & \leq \rho(C\|y_1-y_2\| ).
    \end{align*}
  In view of the Arzel\`{a}--Ascoli theorem, this implies that $ \mathcal{F}$ is relatively compact in $\mathcal{C}(\Omega)$. The second claim follows from this compactness; see \cite[Theorem~2.4.1]{vanderVaart.1996}.
\end{proof}

\begin{lemma}\label{lemma:compactclass2}
Set $\psi_n(\cdot,y) := \int_0^1 \psi''( (1-\lambda)\xi_n(\cdot, y) +\lambda \xi_*(\cdot, y))d\lambda $. The family of functions
  $$
    \int \psi_n(\cdot,y) d\mu(y), \quad n\geq1
  $$
  where $\mu\in \mathcal{C}(\Omega)'$ is any signed measure with $\|\mu\|_{{\rm TV}} \leq 2$, is relatively compact in $\mathcal{C}(\Omega')$.
\end{lemma}

\begin{proof}
    Similarly as in the proof of \cref{lemma:compactclass}, these functions are uniformly bounded and admit a common modulus of continuity.
\end{proof}

We conclude this subsection with a corollary of \cref{pr:three.terms} that will be used to prove the central limit theorem, \cref{th:CLT.pot}. 

\begin{corollary}\label{Coro:exchangeGamman}
 The sequence 
$\Delta_n := [\tilde\Gamma - \tilde\Gamma_n](f_*, g_*) - [\tilde\Gamma  - \tilde\Gamma_n](f_n, g_n)$
satisfies
  $$ \| \Delta_n\|_{\oplus}=o_{\mathbb{P}}\left(\|(f_*-f_n,g_*-g_n)\|_\oplus + n^{-\frac{1}{2}}\right).$$
\end{corollary}

\begin{proof}%
  We write $\Delta_n$ as in~\eqref{eq:DeltaSplit} and detail the proof for the first component only. As in~\eqref{DifferenceGammas}, 
  \begin{align}\label{eq:Delta1explicit}
 \Delta_n^{(1)}=  \int \psi_n(\cdot,y) (f_*-f_n+g_*(y)-g_n(y)) d(Q-Q_n)(y).
\end{align}
Recall the notation from \cref{pr:three.terms}, in particular $ U_n, V_n, W_n\in \mathcal{C}(\Omega)$ and $ U_n', V_n', W_n'\in \mathcal{C}(\Omega')$, and the compact sets $\mathcal{K}$ and $ \mathcal{K}' $. Inserting~\eqref{eq:three.terms} into~\eqref{eq:Delta1explicit}, we have
\begin{multline*}
    \Delta_n^{(1)}=  \underbrace{\|(f_*-f_n,g_*-g_n)\|_\oplus\int \psi_n(\cdot,y) (U_n+U'_n(y)) d(Q-Q_n)(y)}_{=:A_n}\\+ \underbrace{\|(f_*-f_n,g_*-g_n)\|_\oplus\int \psi_n(\cdot,y) (V_n+V'_n(y)) d(Q-Q_n)(y)}_{=:B_n}\\
    + \underbrace{ n^{-\frac{1}{2}}\int \psi_n(\cdot,y) (W_n+W'_n(y)) d(Q-Q_n)(y)}_{=:C_n}.
\end{multline*}
As $\psi_n(x, \cdot)$ is an equicontinuous and bounded family,
$$ \mathcal{F}:= \left\{ y \mapsto \psi_n(x, y) (f(x) + g(y)) : (f, g) \in \mathcal{K} \times \mathcal{K}', \, x \in \Omega, \, n\geq1 \right\} $$
is relatively compact in $\mathcal{C}(\Omega')$. 
Thus $\mathcal{F}$ is uniformly Glivenko--Cantelli and we get 
\begin{align}\label{eq:boundAn}
\begin{split}
    \| A_n\|_\infty & \leq   \|(f_*-f_n,g_*-g_n)\|_\oplus \, \sup_{g\in \mathcal{F}} \left\vert  \int g (y) d(Q-Q_n)(y) \right\vert\\
    &= o_{\mathbb{P}}\left(\|(f_*-f_n,g_*-g_n)\|_\oplus\right).
\end{split}
\end{align}
Next, recalling $ \|(V_n,V_n')\|_\infty\xrightarrow{\text{a.s.}} 0$, the bound
\begin{align*}
     \|B_n\|_\infty&\leq \|(f_*-f_n,g_*-g_n)\|_\oplus \|\psi_n\|_\infty \|(V_n,V_n')\|_\infty   \,\sup_{\|g\|_{\infty}\leq 1}\left\vert \int g (y) d(Q-Q_n)(y) \right\vert\\
     &\leq 2\|(f_*-f_n,g_*-g_n)\|_\oplus \|\psi_n\|_\infty \|(V_n,V_n')\|_\infty   
\end{align*}
yields 
\begin{align}\label{eq:boundBn}
    \| B_n\|_\infty & = o_{\mathbb{P}}\left(\|(f_*-f_n,g_*-g_n)\|_\oplus\right).
\end{align}
Finally, as the unit ball wrt.\ $\|\cdot\|_{\mathbb{F},\beta}$ is compact in $\mathcal{C}(\Omega)\times \mathcal{C}(\Omega')$ and $\mathbb{L}$ is a continuous bijection, the  class 
$$ \mathcal{F}':= \left\{ y \mapsto \psi_n(x, y) (f(x) + g(y)) : \|\mathbb{L}(f, g)\|_{\mathbb{F},\beta}\leq 1, \, x \in \Omega , \, n\geq1\right\} $$
is uniformly Glivenko--Cantelli. Since $\mathbb{L}(W_n,W_n')$ is tight in $\mathcal{B}_{\mathbb{F},\beta}$, $\|\mathbb{L}(W_n,W'_n)\|_{\mathbb{F},\beta}$ is stochastically bounded and we deduce
\begin{align}\label{eq:boundCn}
\begin{split}
     \|C_n\|_\infty &\leq n^{-\frac{1}{2}} \|\mathbb{L}(W_n,W'_n)\|_{\mathbb{F},\beta} \, \sup_{g\in \mathcal{F}'} \left\vert  \int g (y) d(Q-Q_n)(y) \right\vert\leq o_{\mathbb{P}}\left(n^{-\frac{1}{2}}\right). 
\end{split}
\end{align}
Combining the bounds \eqref{eq:boundAn}, \eqref{eq:boundBn} and \eqref{eq:boundCn}, we conclude 
$$ \| \Delta_n^{(1)}\|_{\infty}=o_{\mathbb{P}}\left(\|(f_*-f_n,g_*-g_n)\|_\oplus + n^{-\frac{1}{2}}\right).$$
The result follows since the denominator in~\eqref{eq:DeltaSplit} is continuous and bounded away from zero.
\end{proof}
\subsection{Deriving the central limit theorems}\label{se:proof.CLT.end}

We can now establish the central limit theorems. 
We begin with the one for the potentials.

\begin{proof}[Proof of \cref{th:CLT.pot}]
Armed with the results of \cref{se:consistency,se:key.decomposition}, the central limit theorem for the potentials follows by the standard reasoning in $Z$-estimation; see \cite[Theorem~3.3.1]{vanderVaart.1996}. We give the details for the sake of completeness. We can write~\eqref{eq:development1} as
    \begin{align*}
        \mathbb{L} (f_n-f_*,g_n-g_*) 
        &= [\tilde\Gamma-\tilde\Gamma_n](f_n,g_n) + \|(f_n-f_*,g_n-g_*)\|_{\oplus} (V_{n,\mathbb{L}},V_{n,\mathbb{L}}')\\
        &=  [\tilde\Gamma-\tilde\Gamma_n](f_*,g_*) - \Delta_n  + \|(f_n-f_*,g_n-g_*)\|_{\oplus} (V_{n,\mathbb{L}},V_{n,\mathbb{L}}').
    \end{align*}
    In view of \cref{lemma:CLTInHolder} and \cref{Coro:exchangeGamman}, we deduce that
    $$
      \|\mathbb{L} (f_n-f_*,g_n-g_*)\|_\oplus = o_{\mathbb{P}}\left(\|(f_*-f_n,g_*-g_n)\|_\oplus \right) + \mathcal{O}_{\mathbb{P}}\left(n^{-\frac{1}{2}}\right).
    $$
    Recalling that $\mathbb{L}$ has a bounded inverse (\cref{pr:Invertible}), it follows that
    \begin{equation*}
   \|(f_n-f_*,g_n-g_*)\|_\oplus \leq %
   o_{\mathbb{P}}\left(\|(f_*-f_n,g_*-g_n)\|_\oplus \right) + \mathcal{O}_{\mathbb{P}}\left(n^{-\frac{1}{2}}\right).
\end{equation*}
    As we also have $\|(f_n-f_*,g_n-g_*)\|_\oplus=o_{\mathbb{P}}(1)$ by \cref{lemma.A.s}, this implies the rate
    $\|(f_n-f_*,g_n-g_*)\|_\oplus=\mathcal{O}_{\mathbb{P}}\left(n^{-\frac{1}{2}}\right)$. Substituting this into~\eqref{eq:development1}, we obtain
$$ \left\|\tilde\Gamma(f_n, g_n) - \tilde\Gamma_n(f_n, g_n) - \mathbb{L} (f_n - f_*, g_n - g_*)\right\|_\oplus = o_{\mathbb{P}}\left(n^{-\frac{1}{2}}\right)$$
which, using again \cref{Coro:exchangeGamman} and $\|(f_n-f_*,g_n-g_*)\|_\oplus=o_{\mathbb{P}}(1)$, leads to
\begin{align}\label{eq:Gamma.tilde.vs.L}
\left\|\tilde\Gamma(f_*, g_*) - \tilde\Gamma_n(f_*, g_*) - \mathbb{L} (f_n - f_*, g_n - g_*)\right\|_\oplus = o_{\mathbb{P}}\left(n^{-\frac{1}{2}}\right). \end{align}
Denote by $\tilde{\mathbf{G}}$ the right-hand side in \cref{lemma:CLTInHolder}. As $n^{1/2}\big(\tilde\Gamma(f_*, g_*) - \tilde\Gamma_n(f_*, g_*)\big) \xrightarrow{w} \tilde{\mathbf{G}}$ in $\Banach$ by \cref{lemma:CLTInHolder}, \eqref{eq:Gamma.tilde.vs.L} implies that 
$$
n^{1/2}\,\mathbb{L}(f_n - f_*, g_n-g_*) \xrightarrow{w} \tilde{\mathbf{G}}
$$
in $\Banach$ as well.
Using that $\mathbb{L}^{-1}$ is continuous (\cref{pr:Invertible}) and applying the continuous mapping theorem, we conclude that
$$
n^{1/2}\,(f_n - f_*, g_n-g_*) \xrightarrow{w} \mathbb{L}^{-1}\tilde{\mathbf{G}},
$$
which was the claim.
\end{proof}

We continue with the central limit theorem for the optimal cost.

\begin{proof}[Proof of \cref{th:CLT.cost}]  
    From the fact that $(f_*, g_*)$ and $(f_n, g_n)$ are optimizers of the dual problem~\eqref{eq:dual1} for $(P,Q)$ and $(P_n,Q_n)$, respectively, we deduce the two inequalities
\begin{multline*}
    \ROT(P_n,Q_n)-\ROT(P,Q)\\
    \geq \int\big\{ 
    f_*(x) + g_*(y) - \psi\left({f_*(x)+g_*(y)-c(x,y)}\right) 
    \!\big\}d(  (P_n \otimes   Q_n)- (P \otimes   Q) )(x,y) 
\end{multline*}
and 
\begin{multline*}
    \ROT(P_n,Q_n)-\ROT(P,Q)\\
    \leq \int\big\{ 
    f_n(x) + g_n(y) - \psi\left({f_n(x)+g_n(y)-c(x,y)}\right) 
    \!\big\}d(  (P_n \otimes   Q_n)- (P \otimes   Q) )(x,y). 
\end{multline*}
Using our shorthands~\eqref{eq:xi.def} and dropping the integration variables from the notation, the difference between the upper and the lower bound is
    \begin{multline}\label{eq:ExchangeEmpPopCoroProof}
         \int\big\{ f_n(x) + g_n(y) - \psi\left({f_n(x)+g_n(y)-c(x,y)} \right)\\
        -
    \left(f_*(x) + g_*(y) - \psi\left({f_*(x)+g_*(y)-c(x,y)}\right) \right)
    \!\big\}d(  (P_n \otimes   Q_n)- (P \otimes   Q) )(x,y)\\
    =\int\big\{ (\xi_n - \xi_*) - (\psi(\xi_n) - \psi(\xi_*))\!\big\}d(  (P_n \otimes   Q_n)- (P \otimes   Q)).
    \end{multline}
Using the fundamental theorem of calculus in the same way as  below~\eqref{eq:DeltaSplit}, and setting $\tilde{\psi}_n := 1+ \int_0^1 \psi'\big( (1-\lambda)\xi_n +\lambda \xi_*\big)d\lambda$, the right-hand side of \eqref{eq:ExchangeEmpPopCoroProof} can be written as
    \begin{align*}%
    \int \tilde{\psi}_n \cdot (\xi_n - \xi_*) d(  (P_n \otimes   Q_n)- (P \otimes   Q) )
    \end{align*}
    which has the same structure as~\eqref{eq:Delta1explicit}, except that now the integration is over both variables. We proceed as below~\eqref{eq:Delta1explicit}: insert the decomposition~\eqref{eq:three.terms}, split the integral into three terms, and treat them with the same arguments as in~\cref{eq:boundAn,eq:boundBn,eq:boundCn}. The result is that
   \begin{align*}%
    \eqref{eq:ExchangeEmpPopCoroProof} = o_{\mathbb{P}}\left(\|(f_*-f_n,g_*-g_n)\|_\oplus + n^{-\frac{1}{2}}\right).
    \end{align*}    
    From \cref{th:CLT.pot} we further know that $\|(f_*-f_n,g_*-g_n)\|_\oplus=\mathcal{O}_{\mathbb{P}} \big(n^{-\frac{1}{2}} \big)$, whence we conclude that $\eqref{eq:ExchangeEmpPopCoroProof} = o_{\mathbb{P}} \big(n^{-\frac{1}{2}}\big)$. 
    Recall that \eqref{eq:ExchangeEmpPopCoroProof} is the difference of the upper and lower bounds for $\ROT(P_n,Q_n)-\ROT(P,Q)$. As a consequence, $\ROT(P_n,Q_n)-\ROT(P,Q)$ equals the lower bound up to $o_{\mathbb{P}} \big(n^{-\frac{1}{2}}\big)$; to wit,
    \begin{multline*}
    \ROT(P_n,Q_n)-\ROT(P,Q) +o_{\mathbb{P}}\left(n^{-\frac{1}{2}}\right)\\
   = \int\big\{ 
    \left(f_*(x) + g_*(y) - \psi\left({f_*(x)+g_*(y)-c(x,y)}\right) \right)
    \!\big\}d(  (P_n \otimes   Q_n)- (P \otimes   Q) )(x,y). %
\end{multline*}
We can now apply the central limit theorem for $U$-statistics \cite[Theorem~12.6]{Vaart.1998} to derive the result. (In fact, to obtain the formula for the variance $\sigma^2$, it is easier to directly use the formula of the H\'ajek projection stated just before \cite[Theorem~12.6]{Vaart.1998}.)
\end{proof}

Finally, we show the central limit theorem for the couplings.

\begin{proof}[Proof of \cref{th:CLT.plans}]
Note that 
$$ \int \eta d(\pi_n - \pi)= \int \bar{\eta} d(\pi_n - \pi).  $$
Recalling from \cref{pr:ROTprelims} that $\frac{d\pi}{dP\otimes Q} = \psi'(\xi_*)$ and $\frac{d\pi_n}{dP_n \otimes Q_n} = \psi'(\xi_n)$, we obtain
\begin{align*}
    \int \eta d(\pi_n - \pi)= &\underbrace{\int \bar{\eta}\{ \psi'(\xi_n)-\psi'(\xi_*)\} d((P_n\otimes Q_n)-(P\otimes Q))}_{=:A_n} \\
   & + \underbrace{\int \bar{\eta} \{ \psi'(\xi_n)-\psi'(\xi_*)\} d(P\otimes Q)}_{=:B_n} \\
   & + \underbrace{\int \bar{\eta}\psi'(\xi_*)  d((P_n\otimes Q_n)-(P\otimes Q))}_{=:C_n}.
\end{align*}
Arguing as in \eqref{eq:boundAn} and then using \cref{th:CLT.pot},  we have
$$
A_n= o_{\mathbb{P}}\left(\|(f_*-f_n,g_*-g_n)\|_\oplus\right) =o_{\mathbb{P}}\left(n^{-\frac{1}{2}}\right).
$$
Turning to $B_n$, a first-order Taylor development of $ \psi'$ yields 
\begin{align*}
B_n 
&= \int  ( \xi_n-\xi_*)  \bar{\eta}\psi'' (\xi_*)d(P\otimes Q) +o_{\mathbb{P}}(\|(f_n-f_*,g_n-g_*)\|_{\oplus} ) \\
& = \int  ( \xi_n-\xi_*)  \bar{\eta}\psi'' (\xi_*)d(P\otimes Q) + o_{\mathbb{P}}\left(n^{-\frac{1}{2}}\right).
\end{align*}
Moreover, we see from~\eqref{eq:Gamma.tilde.vs.L} that
\begin{align*}
\left\|(f_n - f_*, g_n - g_*) - \mathbb{L}^{-1}\left( [\tilde\Gamma - \tilde\Gamma_n](f_*, g_*) \right) \right\|_\oplus = o_{\mathbb{P}}\left(n^{-\frac{1}{2}}\right).
\end{align*}
After inserting the definitions of $\tilde\Gamma$ and $\tilde\Gamma_n$, this yields the third equality in
\begin{align*}
    B_n&= \int   ( \xi_n-\xi_*)\bar{\eta} \psi''(\xi_*)  d(P\otimes Q) +o_{\mathbb{P}}\left(n^{-\frac{1}{2}}\right)\\
    &=  \int  ( \oplus(f_n-f_*, g_n-g_*))\bar{\eta} \psi''(\xi_*) d(P\otimes Q) +o_{\mathbb{P}}\left(n^{-\frac{1}{2}}\right)\\
    &= -\int    \oplus\left( \mathbb{L}^{-1}\left[\left( \begin{array}{c}
           \frac{\int \psi'(\xi_*( \cdot ,y)) d(Q_n-Q)(y) } {\int \psi''(\xi_*( \cdot ,y)) dQ(y)}   \\
            \frac{\int \psi'(\xi_*(x, \cdot ))d(P_n-P)(x)}{\int \psi''(\xi_*(x, \cdot ))  dP(x)}  
        \end{array}\right)\right]_{\oplus} \right)\bar{\eta} \psi''(\xi_*) d(P\otimes Q) +o_{\mathbb{P}}\left(n^{-\frac{1}{2}}\right)\\
        &= -\frac{1}{n^2} \sum_{i,j=1}^n \int \bigg\{   \oplus\left( \mathbb{L}^{-1}\left[\left( \begin{array}{c}
           \frac{ \psi'(\xi_*( \cdot ,Y_j)) - \int \psi'(\xi_*( \cdot ,y))dQ (y) } {\int \psi''(\xi_*( \cdot ,y)) dQ(y)}   \\
            \frac{ \psi'(\xi_*(X_i, \cdot )-\int \psi'(\xi_*(x, \cdot ))dP(x)) }{\int \psi''(\xi_*(x, \cdot )) dP(x)}  
        \end{array}\right)\right]_{\oplus} \right)\bar{\eta} \psi'' (\xi_*) \bigg\} d(P\otimes Q) \\
        & \quad\,+o_{\mathbb{P}}\left(n^{-\frac{1}{2}}\right) =: \bar{B}_n +o_{\mathbb{P}}\left(n^{-\frac{1}{2}}\right).
\end{align*}
Note also that $\bar{B}_n$ is centered (w.r.t.\ $\mathbb{P}$) because the precomposition of a centered random variable with a bounded linear operator is again centered. Turning to $C_n$, we can write 
\begin{align*}
   C_n= \frac{1}{n^2}\sum_{i,j=1}^n \bar{\eta}(X_i,Y_j)\psi'(\xi_*(X_i,Y_j)) \,-\,\int \bar{\eta}\psi'(\xi_*)  d(P\otimes Q)
\end{align*}
which is also centered. Moreover, the variances of $\bar{B}_n$ and $C_n$ are bounded by a constant depending only on $\|\xi_*\|_{\infty}$, $ \|\eta\|_{\infty}$,
$$ \left\|\left(\int \psi''(\xi_*( \cdot ,y))  dQ(y)\right)^{-1}\right\|_{\infty}\quad {\rm and} \quad  \left\|\left(\int \psi''(\xi_*( x,\cdot ))  dP(x)\right)^{-1}\right\|_{\infty}.$$
In summary, up to negligible terms,  
$\int \eta d(\pi_n - \pi)$ is a $U$-statistic with finite variance. We can then apply the central limit theorem for  $U$-statistics (see \cite[Theorem~12.6]{Vaart.1998}) to conclude the result. (Once again, to obtain the formula for the variance $\sigma^2(\eta)$, it is easier to directly use the formula of the H\'ajek projection stated just before \cite[Theorem~12.6]{Vaart.1998}.)
\end{proof}

\begin{funding}
SE is grateful for support by the German Research Foundation through Project 553088969 as well as the Cluster of Excellence “Machine Learning --- New Perspectives for Science” (EXC 2064/1 number 390727645). MN acknowledges funding by NSF Grants DMS-2407074 and DMS-2106056.
\end{funding}

\bibliographystyle{imsart-number}
\bibliography{Biblio}

\pagebreak[4]
\appendix

\section{Simulations}\label{sec:simulations}
In this section we provide numerical examples to illustrate our results, and circumstantial evidence that the obtained results may extend to quadratic regularization. We recall from \cref{ex:divergences} that quadratic regularization is a boundary case of \cref{as:divergence}.

For our experiments we need to solve the population problem very accurately. In EOT, one typically uses quadratic cost with Gaussian marginals in such situations, as that example allows for a closed-form solution (e.g., \cite{delbarrio.2020.statisticaleffectentropicregularization,Mallasto.2021}). The first subsection below details the only example with semi-closed form solution for more general divergences that we are aware of. The second subsection uses that example to numerically study the convergence rate for two divergences.

\subsection{Example with semi-explicit population potentials}\label{sec:explicit}
Let $\mathbb{T}^d=\R^d/\Z^d$ be the flat torus and let $\mu$ be the volume measure on $\mathbb{T}^d$. 
Denote by $[x]$ the equivalence class of a vector $x\in \R^d$. 
The distance on $\mathbb{T}^d$ is given by 
$$ d([x],[y])=\inf_{z\in \Z^d}\|x-y-z\| .$$
We consider the marginals $P=Q=\mu$ and the transport cost $c([x],[y])=d^2([x],[y])$. In this symmetric ``self-transport'' problem, the population potentials $(f_*,g_*)$ can be chosen such that $g_*=f_*$; this choice also removes the ambiguity about an additive constant. Indeed, $f_*$ is uniquely determined by the equation
$$ 1=\int \psi'\left(\frac{f_*([x])+f_*([y])-d^2([x],[y])}{\eps}\right) d\mu([y]), \quad  x \in \mathbb{R}^d.$$
Define the constant $C_{\eps,d,\psi}\in\mathbb{R}$ via
$$ 1=\int \psi'\left(\frac{2C_{\eps,d,\psi}-d^2([0],[y])}{\eps}\right) d\mu([y]).$$
For fixed $x\in \R^d$, consider the  translation 
$ y\mapsto T_x(y) := x+y$ on $\mathbb{R}^d$, which induces the map
$$ [y]\mapsto S_{x}([y]) := [T_x(y)] =[x+y]$$
on the quotient space $\mathbb{T}^d=\R^d/\Z^d$.  Note that  $S_{x}$  is an isometry
 and that the ``translation invariance'' of the volume measure implies $({S_x})_\#\mu=\mu$. Hence 
\begin{align*}
1&= \int \psi'\left(\frac{2C_{\eps,d,\psi}-d^2([0],[y])}{\eps}\right) d\mu(y)
=\int\psi'\left(\frac{2C_{\eps,d,\psi}-d^2([x],S_x([y]))}{\eps}\right) d\mu(y)\\
 &=\int\psi'\left(\frac{2C_{\eps,d,\psi}-d^2([x],[y])}{\eps}\right) d({S_x})_\#\mu(y)
 =\int \psi'\left(\frac{2C_{\eps,d,\psi}-d^2([x],[y])}{\eps}\right) d\mu(y)
\end{align*}
for all $x\in \R$, showing that $f_*\equiv C_{\eps,d,\psi}$.

\subsection{Numerical illustrations}

While the above example allows for explicit calculation, we emphasize that it actually leads to a degenerate regime, since the variance of the limits is zero. This is most easily seen in \cref{th:CLT.cost}, where the variance is governed by the terms $f_*(X)$ and $\int \psi\left(\frac{f_*(X)+g_*(y)-c(X,y)}{\eps} \right) d\mu(y)$ for $X \sim \mu$ (as well as the symmetric terms for $Y$). We see that both terms are constant due to $f_* = g_* \equiv C_{\varepsilon, d, \psi}$ and the shift invariance of the cost $c = d^2$ and $\mu$. Hence, we expect to see rates which are even faster than $n^{-1/2}$, irrespective of dimension. While ideally we would like to also illustrate the results in non-degenerate examples, we are not aware of further examples where explicit solutions exist, and the latter are required to correctly visualize the asymptotic regime. 

To obtain approximate solutions for the empirical marginals $P_n$ and $Q_n$, we use the Sinkhorn-like iterations described for instance in \cite{LorenzMahler.22}. The results are illustrated in Figure~\ref{fig:qotexample}. As predicted by \cref{th:CLT.cost}, we observe rates even faster than $n^{-1/2}$, and the slopes are similar for each dimension. (This is in contrast to upper bounds in \cite{BayraktarEckstein.2025.BJ}, which would only yield $\max\{n^{-1/2}, n^{-2/d}\}$ in this case.) We observe similar slopes for $\varphi(x) = \frac{2}{3}(x^{3/2} -1)$ and $\varphi(x) = \frac{1}{2} (x^2 - 1)$. Note that the former divergence satisfies \cref{as:divergence} as the conjugate is cubic, whereas the latter (quadratic) divergence does not.

\begin{figure}\label{fig:qotexample}
	\begin{minipage}{0.33\textwidth}
		\hspace*{-0.4cm}\includegraphics[width=1.15\textwidth]{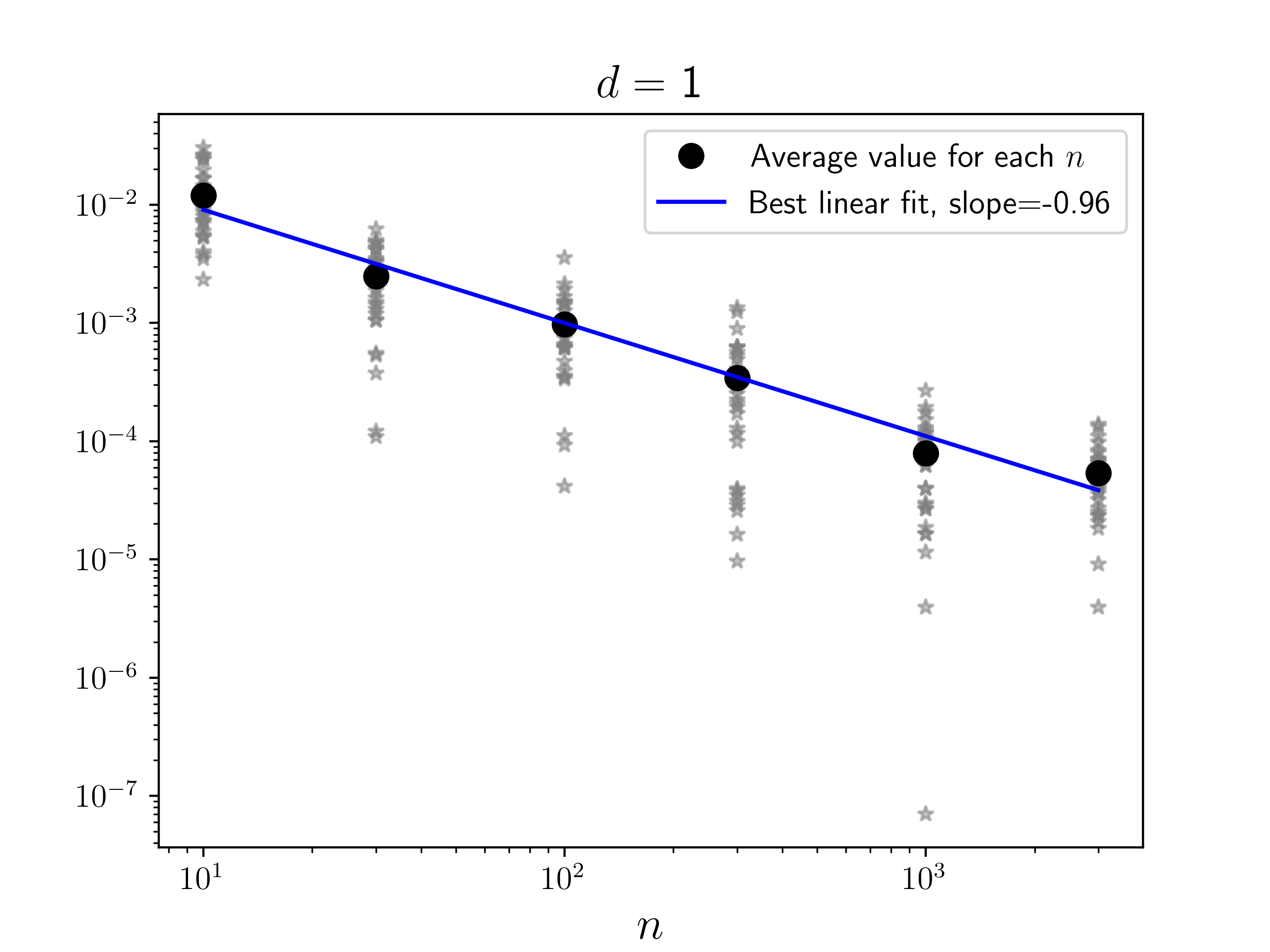}
	\end{minipage}%
	\begin{minipage}{0.33\textwidth}
	\hspace*{-0.1cm}\includegraphics[width=1.15\textwidth]{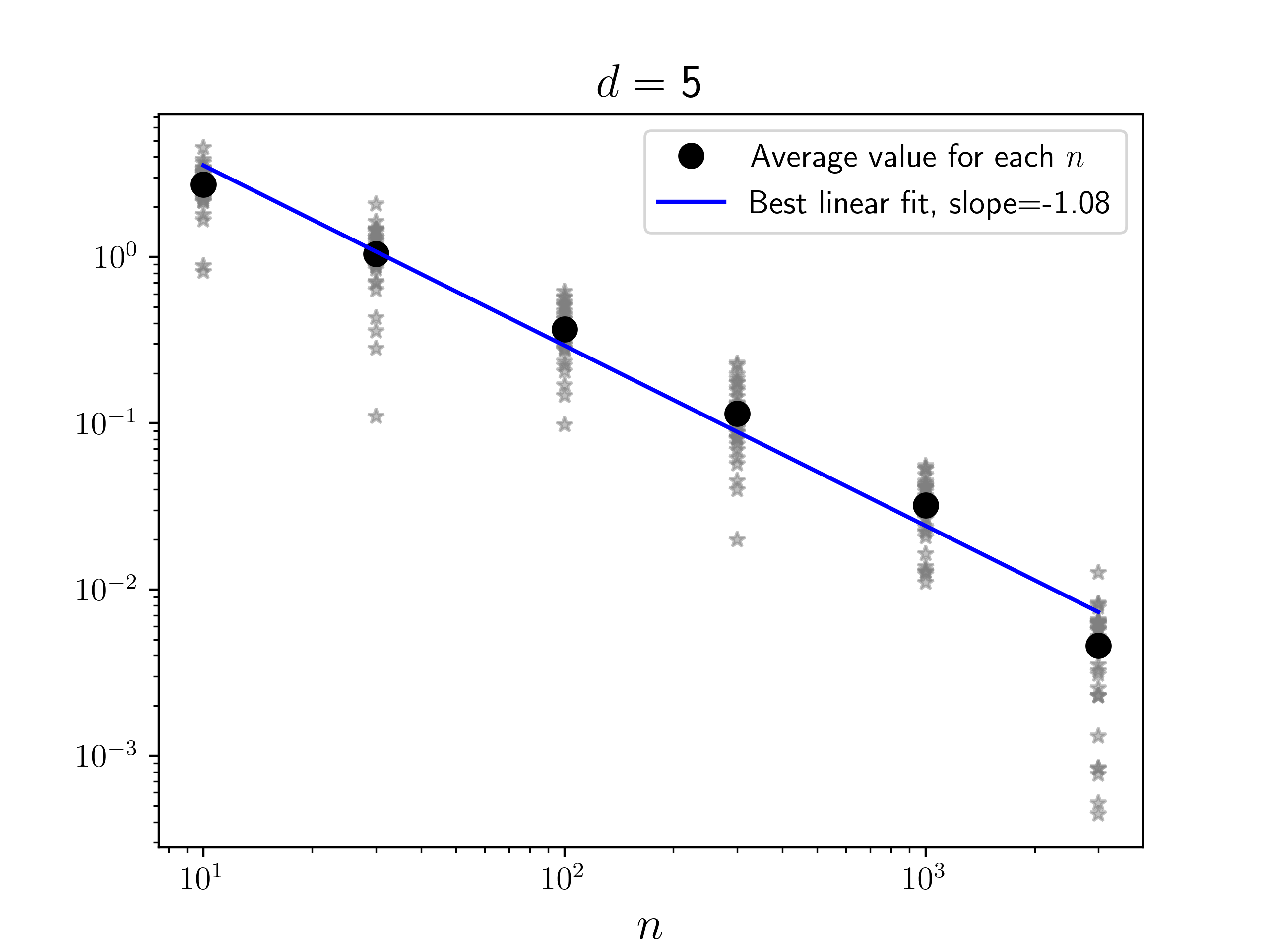}
\end{minipage}%
	\begin{minipage}{0.33\textwidth}
	\hspace*{0.1cm}\includegraphics[width=1.15\textwidth]{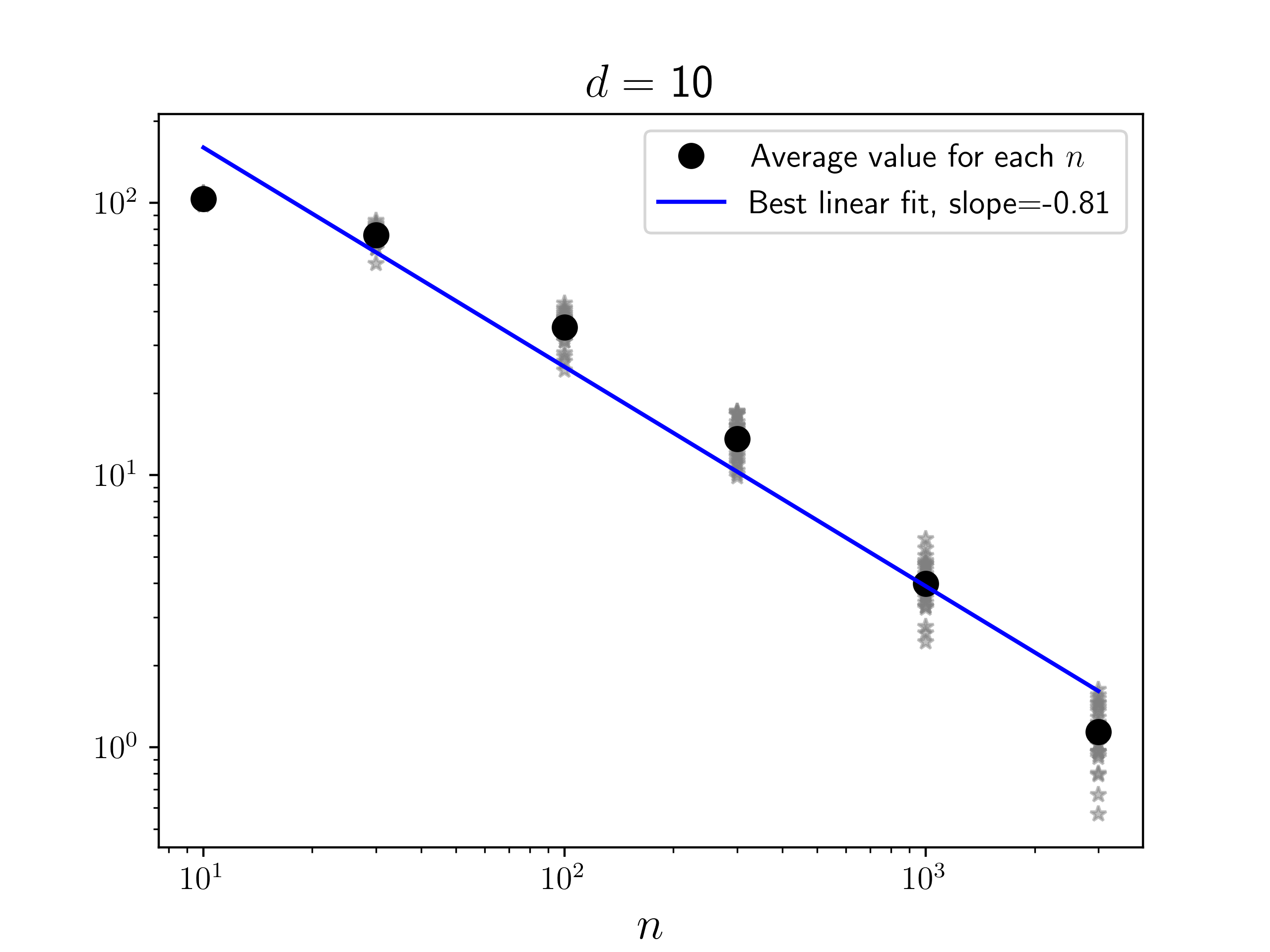}
\end{minipage}\\
	\begin{minipage}{0.33\textwidth}
		\hspace*{-0.4cm}\includegraphics[width=1.15\textwidth]{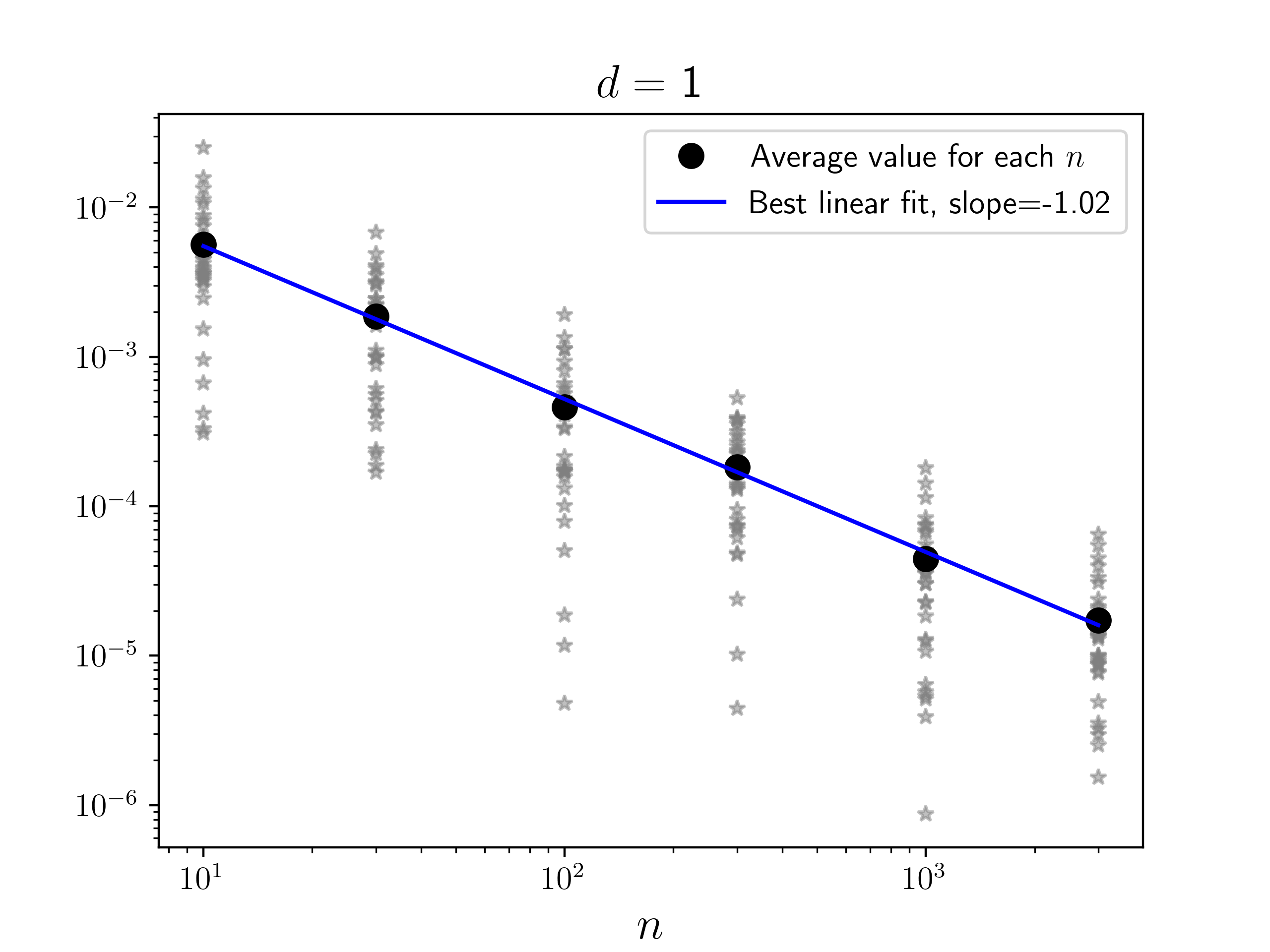}
	\end{minipage}%
	\begin{minipage}{0.33\textwidth}
	\hspace*{-0.1cm}\includegraphics[width=1.15\textwidth]{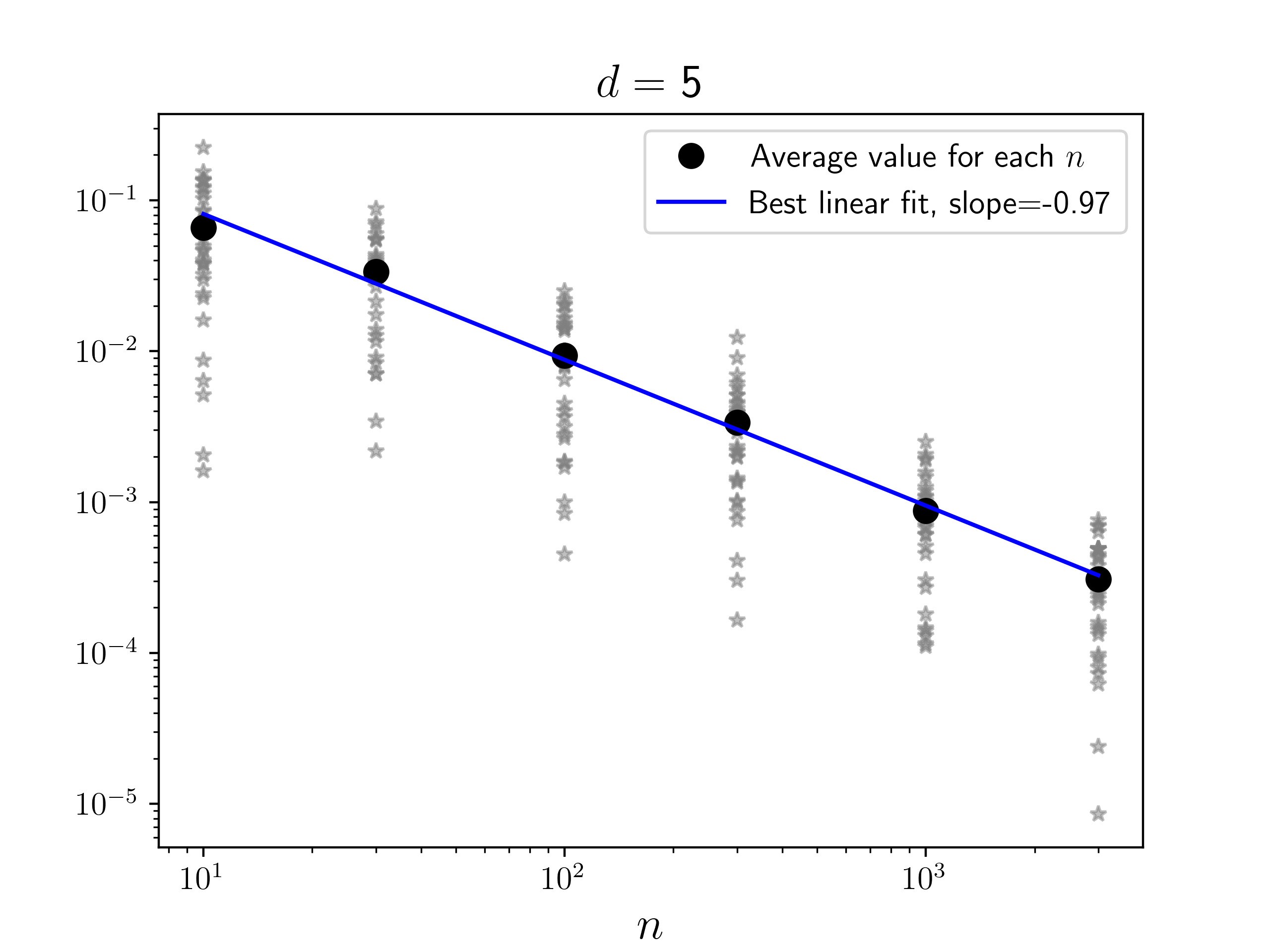}
\end{minipage}%
	\begin{minipage}{0.33\textwidth}
	\hspace*{0.1cm}\includegraphics[width=1.15\textwidth]{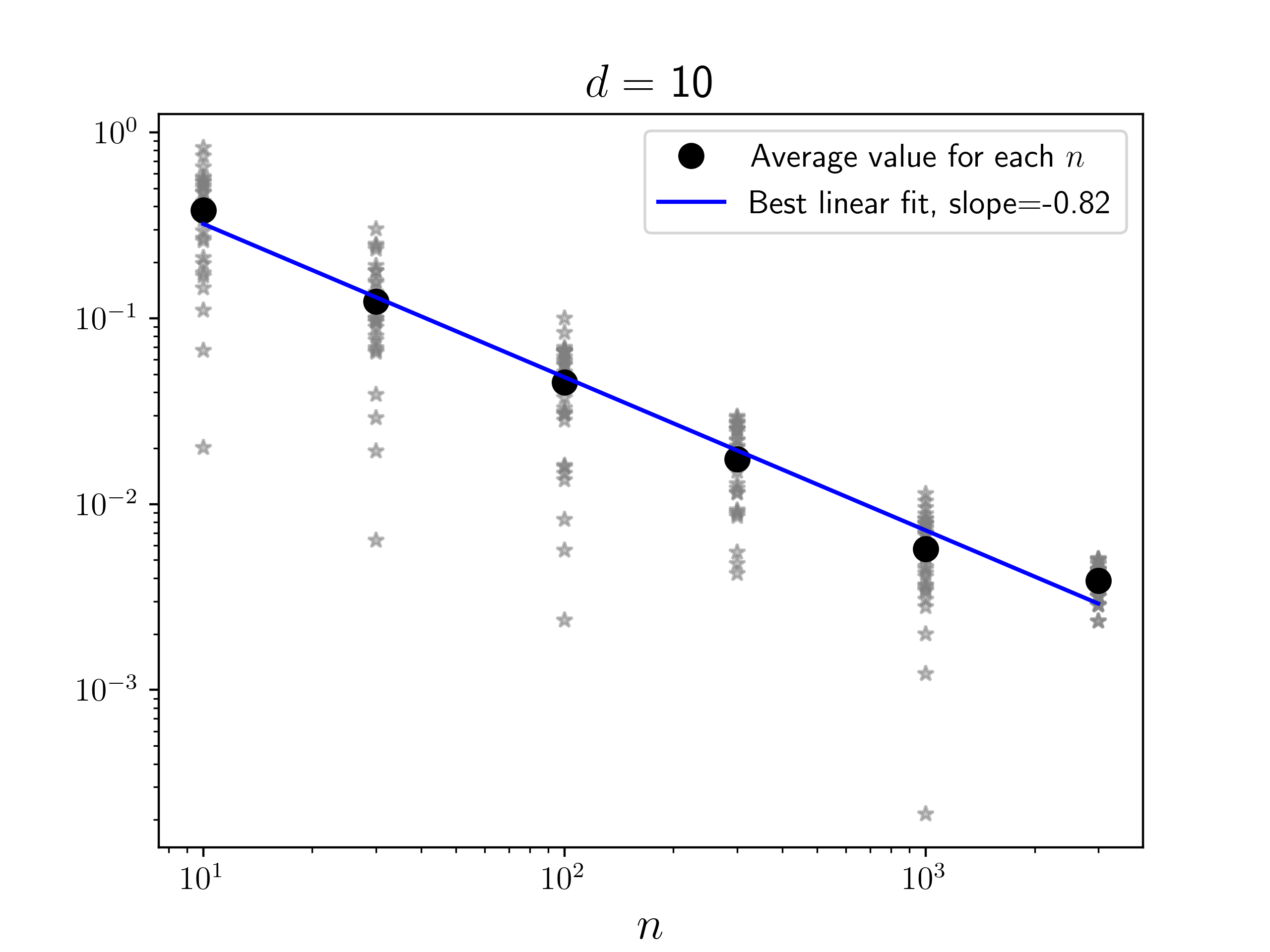}
\end{minipage}%
	\caption{Log-log plots of $\left|\mathrm{OT}_\eps(P_n, Q_n) - \mathrm{OT}_\eps(P, Q)\right|$ for dimensions \mbox{$d \in \{1, 5, 10\}$} (left, middle, right) and two different divergences, namely $\varphi(x) = \frac{2}{3}(x^{3/2} -1)$ leading to $\psi(x) = \frac{1}{3} x_+^3 + \frac{2}{3}$ (top) and $\varphi(x) = \frac{1}{2} (x^2 - 1)$ leading to $\psi(x) = \frac{1}{2} x_+^2 + \frac{1}{2}$ (bottom). We use $\eps=0.5$. We numerically observe approximate rates $n^{\alpha}$ with $\alpha \approx -1$ for all cases, irrespective of dimension. The approximation is based on thirty independent evaluations of $\left|\mathrm{OT}_\eps(P_n, Q_n) - \mathrm{OT}_\eps(P, Q)\right|$ illustrated by the gray dots, while the black dots are the respective averages for each $n\in\{10, 30, 100, 300, 1000, 3000\}$.}
\end{figure}

\section{Central Limit Theorem in H\"older Space}\label{se:HolderCLT}

The aim of this section is to provide a central limit theorem in the H\"{o}lder space $\mathcal{C}^{0,\beta}(\Omega)$ with $0\leq \beta< 1$ for  
an i.i.d.\ sequence of random variables with values in $\mathcal{C}^{0,1}(\Omega)$ and finite second order moments. While it is possible that this result is known, we have not found a suitable reference in the literature. We first recall some preliminaries on type II spaces, following the excellent survey \cite{Naresh.CLT.1976}. 

\begin{definition}
    Let $(E, \|\cdot\|_E)$ and $(F, \|\cdot\|_F)$ be separable Banach spaces such that $E \subset F$. The pair $ (E,F)$ is said to be of type~II if there exists $C\geq 0$ such that  
    $$ \mathbb{E}\left[\left\|\sum_{i=1}^n X_i \right\|_{F}^2 \right] \leq C\cdot \sum_{i=1}^n\mathbb{E}\left[\left\| X_i \right\|_E^2 \right]$$
    for any $n\in\mathbb{N}$ and any sequence $X_1, \dots, X_n$ of independent $E$-valued random variables with $\mathbb{E}[\|X_i\|_E]<\infty$ for all $i=1, \dots, n$. 
\end{definition}
The type II property is fundamental to derive central limit theorems 
 in Banach spaces (see \cite{Naresh.CLT.1976} and the references therein). To state the central limit theorem, recall that the expectation $\mathbb{E}[X]$ of a random variable $X$ with $\mathbb{E}[\|X\|_E]<\infty$ in a Banach space $E$ is defined as the following element of the bidual $E'' $: 
$$ E'  \ni g\mapsto \mathbb{E}[g(X)]\in \R,$$
where $E'$ denotes the dual of $E$. A measure on $E$ is called Gaussian if its finite-dimensional distributions are Gaussian; i.e., given $g_1,\dots,g_n\in E'$, the distribution on $\mathbb{R}^n$ induced by $(g_1,\dots,g_n)$ is Gaussian. 
A random variable $\mathbf{G}\in E $ is called Gaussian if its distribution is Gaussian.

\begin{theorem}[{\cite[Theorem 3.1]{Naresh.CLT.1976}}]\label{pr:type2.CLT}
    Let $ (E,F)$ be of type~II. Then  every sequence $\{X_n\}_{n} \subset E$ of i.i.d.\ random vectors with $\mathbb{E}[\|X_i\|_E^2]<\infty$ and $  \mathbb{E}[X_i]=0$ satisfies the central limit theorem in $E$; i.e., there exists a centered Gaussian random variable $ \mathbf{G}\in E $ such that 
    $$ \sqrt{n}\left( \frac{1}{n}\sum_{i=1}^n X_i\right)\xrightarrow{w} \mathbf{G} \quad\text{in}\quad E.$$
\end{theorem}

To establish the desired central limit theorem, we need to show that the Banach pair $(\mathcal{C}^{0,1}(\Omega), \mathcal{C}^{0,\beta}(\Omega))$ is of type II. To that end, we first state an alternative characterization and a sufficient condition for being of type II. Recall that a Rademacher sequence $ \{\epsilon_n\}_{n}$ is a sequence of i.i.d.\ (scalar) random variables with 
$\mathbb{P}(\epsilon_n=1)=\mathbb{P}(\epsilon_n=-1)=\frac{1}{2}.$

\begin{proposition}[{\cite[Theorems 3.1 and 2.3]{Naresh.CLT.1976}}]\label{pr:type2.sufficient}
    Let $(E, \|\cdot\|_E)$ and $(F, \|\cdot\|_F)$ be separable Banach spaces such that $E \subset F$. The following are equivalent:
    \begin{enumerate}
        \item[(i)] The pair $ (E,F)$ is of type~II,
        \item[(ii)] for any sequence $\{x_n\}_n\subset E$ such that $\sum_{n=1}^\infty\| x_n\|_E^2<\infty$, and any Rademacher sequence $ \{\epsilon_n\}_{n}$, the sum $\sum_{n=1}^\infty  x_n \epsilon_n $ converges a.s.\ in $F$.
    \end{enumerate}
    Moreover, the following is a sufficient condition for (i) and (ii):
    \begin{enumerate}
        \item[(iii)] for any sequence $\{x_n\}_n\subset E$ such that $\sum_{n=1}^\infty\| x_n\|_E^2<\infty$, and any i.i.d.\ sequence $ \{\eta_n\}_{n}$ of standard Gaussian random variables, the sum $\sum_{n=1}^\infty  x_n \eta_n $ converges a.s.\ in~$F$.
    \end{enumerate}
\end{proposition}

We can now state and prove the desired central limit theorem in $\mathcal{C}^{0,\beta}(\Omega)$.

\begin{theorem}\label{th:CLT.Holder}
    Let $\beta\in [0,1)$ and let $\Omega$ be a compact subset of $\mathbb{R}^d$. Then the pair $(\mathcal{C}^{0,1}(\Omega),\mathcal{C}^{0,\beta}(\Omega))$ is of type II. As a consequence, any i.i.d.\ sequence of random functions $\{X_n\}_n $ in $\mathcal{C}^{0,1}(\Omega)$  with $\mathbb{E}[\|X_1\|^2_{0,1}]<\infty$ and $\mathbb{E}[X_1]=0$ satisfies the central limit theorem in $\mathcal{C}^{0,\beta}(\Omega)$; i.e., there exists a centered Gaussian random variable $ \mathbf{G}\in \mathcal{C}^{0,\beta}(\Omega)$ such that 
    $$ \sqrt{n}\left( \frac{1}{n}\sum_{i=1}^n X_i\right)\xrightarrow{w} \mathbf{G} \quad\text{in}\quad \mathcal{C}^{0,\beta}(\Omega).$$
    Note that, necessarily, $\mathbb{E}[\mathbf{G}(x)\mathbf{G}(x')]=\mathbb{E}[X_1(x)X_1(x')]$ for all $x,x'\in \Omega$.
\end{theorem}

\begin{proof}
    Let $\{f_n\}_{n}\subset \mathcal{C}^{0,1}(\Omega)$ be a sequence of functions with
$\sum_{n=1}^\infty \|f_n\|_{0,1}^2 <\infty $
and let $ \{\eta_n\}_{n \in \mathbb{N}}$ be an i.i.d.\ sequence of (scalar) standard Gaussian random variables. In view of \cref{pr:type2.CLT} and \cref{pr:type2.sufficient}, our goal is to show that the partial sum $\sum_{n=1}^N  f_n \eta_n $ converges a.s.\ in $\mathcal{C}^{0,\beta}(\Omega)$ to some limit $\mathbb{G}$.

Note that for every $x\in \Omega$, the partial sum
$\mathbb{G}_N(x):=\sum_{n=1}^N \eta_n f_n(x)$
converges a.s.\ by Kolmogorov's Three Series Theorem, hence we can define
$\mathbb{G}(x)$ as the a.s.\ limit $\sum_{n=1}^\infty \eta_n f_n(x) $. Thus $\mathbb{G}$ is a scalar, centered Gaussian process indexed by $x\in\Omega$. 
To show that the a.s.\ convergence $\mathbb{G}_N\to \mathbb{G}$ also holds in the norm of $\mathcal{C}^{0,\beta}(\Omega)$, it suffices to show weak convergence in $\mathcal{C}^{0,\beta}(\Omega)$, because by the Lévy--Itô--Nisio theorem \cite[Theorem~2.4]{Ledoux.Talagrand.1991.Book}, these two convergences are equivalent the for partial sum $\mathbb{G}_N=\sum_{n=1}^N \eta_n f_n$ of the symmetric i.i.d.\ random variables $\{\eta_n f_n \}_{n}$. In view of the pointwise convergence, weak convergence is implied by relative compactness (for the weak convergence topology), hence by Prokhorov's theorem \cite[Theorem~2.1]{Ledoux.Talagrand.1991.Book} it suffices to check that given $\delta\in (0,1)$ there exists a compact  $\mathcal{K}\subset\mathcal{C}^{0,\beta}(\Omega)  $ such that
\begin{align}\label{eq:Kolmog.need}
\mathbb{P}( \mathbb{G}_N \in \mathcal{K}) \geq 1-\delta \quad \text{for all $N\geq1$}.
\end{align}
Since $\mathcal{C}^{0,\beta'}(\Omega)$ is compactly embedded in $\mathcal{C}^{0,\beta}(\Omega)$ for $\beta'>\beta$ (e.g., \cite[Lemma~6.33]{GilbargTrudinger.83}), it is enough to show that 
$$ \sup_{N\geq1} \mathbb{E}[ \|\mathbb{G}_N\|_{0,\beta'}^\gamma ] < \infty$$
for some $\beta'>\beta$ and $\gamma\geq 1$. 

To that end, we shall use Kolmogorov's regularity theorem. Recall that two processes $X(x)$ and $\tilde X(x)$ indexed by $x\in \Omega$ are modifications of one another if $\mathbb{P}\{\tilde X(x)=X(x)\}=1$ for all $x\in \Omega$, whereas they are called indistinguishable if $\mathbb{P}\{ \tilde X(x)=X(x)\ \text{for all $x\in \Omega$}\}=1$. Kolmogorov's regularity theorem %
asserts that if $\{G(x)\}_{x \in \Omega}$ is a real-valued stochastic process satisfying  
\begin{align}\label{eq:Kolmog.cond}
  \mathbb{E}[|G(x)-G(y)|^\gamma] \leq C\| x-y\|^{d+\epsilon}    
\end{align}
for some $\gamma\geq1$ and $\epsilon>0$,
then $G$ admits a modification
$\tilde G$ with 
\begin{align}\label{eq:Kolmog.concl}
  \mathbb{E}\left[ \|\tilde G\|_{0,\frac{\epsilon}{\gamma}}^\gamma\right] \leq C_{d,\gamma,\epsilon,C}
\end{align}  
where the constant $C_{d,\gamma,\epsilon,C}$ depends only on $d$, $\gamma$, $\epsilon$ and $C$. In particular, establishing~\eqref{eq:Kolmog.cond} for $G=\mathbb{G}_N$ with constants independent of $N$ implies~\eqref{eq:Kolmog.concl} for a version $\tilde{\mathbb{G}}_N$ with a right-hand side $C_{d,\gamma,\epsilon,C}$ independent of~$N$. Moreover, as we already know that $x\mapsto \mathbb{G}_N(x)$ is continuous, the modification $\tilde{\mathbb{G}}_N$ is necessarily indistinguishable from $\mathbb{G}_N$. In summary, establishing~\eqref{eq:Kolmog.cond} for $G=\mathbb{G}_N$ with constants independent of $N$  implies
\begin{align*}%
  \sup_{N\geq1} \mathbb{E}\left[ \| \mathbb{G}_N\|_{0,\frac{\epsilon}{\gamma}}^\gamma\right] <\infty,
\end{align*}  
which is the desired bound provided that $\beta':=\frac{\epsilon}{\gamma}\in (\beta,1]$.

Fix $\alpha\geq1$ large enough such that setting $\epsilon:=\alpha$ and $\gamma:=d+\alpha$, we have $\beta':=\frac{\epsilon}{\gamma}=\frac{\alpha}{d+\alpha}>\beta$. As explained above, it suffices to show
\begin{align}\label{eq:Kolmog.cond.here}
  \mathbb{E}[|\mathbb{G}_N(x)-\mathbb{G}_N(y)|^\gamma] \leq C\| x-y\|^{\gamma}    
\end{align}
for some $C>0$ independent of $N$, as that will imply \eqref{eq:Kolmog.need} and thus the claim.

By the Marcinkiewicz--Zygmund inequality (e.g., \cite[Theorem~10.3.2]{Chow.1997.Springer}) there exists a constant $C>0$ independent of $N$ such that 
\begin{align*}
    \mathbb{E}[|\mathbb{G}_N(x)-\mathbb{G}_N(y)|^\gamma]&= \mathbb{E}\left[\left|\sum_{n=1}^N \eta_n (f_n(x)-f_n(y))  \right|^\gamma\right]\\  &\leq C \mathbb{E}\left[\left|\sum_{n=1}^N (\eta_n (f_n(x)-f_n(y)))^2  \right|^{\frac{\gamma}{2}}\right]\\
    &\leq C \|x-y\|^{\gamma}  \mathbb{E}\left[\left|\sum_{n=1}^N (\eta_n)^2 \|f_n\|_{0,1}^2  \right|^{\frac{\gamma}{2}}\right].
\end{align*}
We may assume  without loss of generality that $s_N:=\sum_{n=1}^N \|f_n\|_{0,1}^2> 0$. Noting that $s_N\leq s_\infty <\infty$ by the choice of $\{f_n\}_n$, we deduce
\begin{align*}
    \mathbb{E}[|\mathbb{G}_N(x)-\mathbb{G}_N(y)|^\gamma]&\leq C (s_\infty)^{\frac{\gamma}{2}} \|x-y\|^{\gamma}    \mathbb{E}\left[\left(\frac{1}{s_N}\sum_{n=1}^N (\eta_n)^2 \|f_n\|_{0,1}^2  \right)^{\frac{\gamma}{2}}\right].
\end{align*}
Now Jensen's inequality yields
\begin{align*}
    \mathbb{E}[|\mathbb{G}_N(x)-\mathbb{G}_N(y)|^\gamma]&\leq C(s_\infty)^{\frac{\gamma}{2}} \|x-y\|^{\gamma}   \mathbb{E}\left[\frac{1}{s_N}\sum_{n=1}^N (\eta_n)^\gamma \|f_n\|_{0,1}^2  \right]\\
    &=C (s_\infty)^{\frac{\gamma}{2}} \|x-y\|^{\gamma}  \frac{1}{s_N}\sum_{n=1}^N \mathbb{E}[(\eta_n)^\gamma] \|f_n\|_{0,1}^2 \\
    &=C \mathbb{E}[|\eta_1|^\gamma](s_\infty)^{\frac{\gamma}{2}} \|x-y\|^{\gamma}    \frac{1}{s_N}\sum_{n=1}^N  \|f_n\|_{0,1}^2
    \\
    &=C \mathbb{E}[|\eta_1|^\gamma](s_\infty)^{\frac{\gamma}{2}} \|x-y\|^{\gamma},
\end{align*}
which is the desired estimate~\eqref{eq:Kolmog.cond.here}.
\end{proof}

\section{Omitted Proofs}\label{se:proofsAppendix}
\subsection{Proofs for \cref{se:prelims}}\label{se:proofsPrelims}

In this subsection we prove~\cref{pr:ROTprelims} summarizing background results that were used throughout the text. We first state a lemma recalling the analogue of the $c$-conjugate that is standard in optimal transport. For the setting of regularization by $f$-divergence, this notion was introduced by~\cite{DiMarinoGerolin.20b}. (Note, however, that some of the results in~\cite{DiMarinoGerolin.20b} are flawed because the conjugate of~$\varphi$ was taken over $\mathbb{R}$ instead of~$\mathbb{R}_+$, leading to signed measures instead of couplings.)

\begin{lemma}\label{le:conjugate}
    Let $P,Q$ be probability measures on $\mathbb{R}^d$ with supports $\Omega,\Omega'$. Let $c\in\mathcal{C}(\Omega\times\Omega')$ be bounded and have modulus of continuity~$\rho$. Given any bounded measurable function $g:\Omega'\to\mathbb{R}$, there exists a unique function $f:\Omega\to\mathbb{R}$ such that
    \begin{equation}
        \label{eq:conj}
             \int  \psi'(f(x)+ g(y)-c(x,y)) dQ(y)=1\quad \text{for all }x\in\Omega.
    \end{equation}
    Moreover, $f$ is uniformly continuous with modulus $\rho$ and its oscillation is bounded as
    \begin{equation}
        \label{eq:osc.f}
        \sup_{x\in\Omega} f(x)- \inf_{x\in\Omega} f(x) \leq 2\|c\|_\infty,
    \end{equation}    
    while $\inf_{x, y} \{f(x) + g(y)\} \leq t_0 + \|c\|_\infty$ and $\sup_{x, y} \{f(x) + g(y) \}\geq t_0 - \|c\|_\infty$, where~$t_0$ is defined in \cref{as:divergence}.  Finally, $f$ solves the concave optimization
    \begin{equation*}%
     \sup_{f\in  L^\infty(P)} \, \int \big(f\oplus g - \psi(f\oplus g - c)\big) \,d(P\otimes Q).
    \end{equation*} 
\end{lemma}

\begin{proof}[Proof of \cref{le:conjugate}.]
    As $g$ and $c$ are bounded, $\lim_{s\to\infty} \psi'(s+ g(y)-c(x,y))=\infty$ and $\lim_{s\to0} \psi'(s+ g(y)-c(x,y))<1$ by the properties of $\psi$, where the limits are uniform in $(x,y)$. As $\psi'$ is continuous, the intermediate value theorem yields the existence of $f$ solving~\eqref{eq:conj}. Let $x, \tilde{x} \in \Omega$ and assume without loss of generality that $f(\tilde{x}) \leq f(x)$. As $\psi'$ is nondecreasing, \eqref{eq:conj} yields  
    \begin{align*}
    \int \psi'(f(x) + g(y) - c(x, y))dQ(y) &= 1 = \int \psi'(f(\tilde{x}) + g(y) - c(\tilde{x}, y))dQ(y) \\
    &\leq \int \psi'(f(\tilde{x}) + g(y) - c(x, y) + \rho(\|x - \tilde{x}\|))dQ(y).
    \end{align*}
    As $\psi'$ is, in addition, strictly increasing on $[t_0-\delta, \infty)$, this implies $f(x) \leq f(\tilde{x}) + \rho(\|x-\tilde{x}\|)$, showing that $f$ has modulus of continuity $\rho$. The same argument, applied with $x=\tilde{x}$, also shows that $f(x)$ is uniquely determined by \eqref{eq:conj}, and also that the oscillation of $f$ is bounded by the one of $c$:
    \begin{equation}\label{eq:unifBoundPot}
    \sup_{x} f(x) - \inf_{x} f(x) \leq \sup_{x, y} c(x, y) - \inf_{x, y} c(x, y) \leq 2 \|c\|_\infty.
    \end{equation}
    Furthermore, $\psi'(t_0) = 1$ by \cref{as:divergence}. Hence \eqref{eq:conj} implies that 
    \[
    \inf_{x, y} \{f(x) + g(y) - c(x, y)\} \leq t_0 \leq  \sup_{x, y} \{f(x) + g(y) - c(x, y)\}.
    \]
    Thus $\inf_{x, y} \{f(x) + g(y)\} \leq t_0 + \|c\|_\infty$ and $\sup_{x, y} \{f(x) + g(y) \}\geq t_0 - \|c\|_\infty$. 
\end{proof}

\begin{proof}[Proof of \cref{pr:ROTprelims}.] (i) We first show $\ROT(P,Q) \geq \DUAL(P,Q)$, the so-called weak duality. The definition of $\psi$ implies that $$\psi(f\oplus g - c) \geq \frac{d\pi}{d(P\otimes Q)}\cdot (f\oplus g-c) -\varphi\left(\frac{d\pi}{d(P\otimes Q)}\right)$$ for any $(f, g) \in L^\infty(P)\times L^\infty(Q)$ and any $\pi \in \Pi(P, Q)$ with $\pi \ll P\otimes Q$. Combining this with $\int f\oplus g \,d\pi = \int f\oplus g \,d(P \otimes Q)$ yields the claimed weak duality. The converse inequality will be shown in (v) below.

    (ii) Primal existence follows directly from the weak compactness of $\Pi(\mu,\nu)$ and the weak lower semi-continuity of the objective.
    
    (iii) To show dual existence, let $(\tilde{f}^n,\tilde{g}^n)_n$ be a maximizing sequence for $\DUAL(P, Q)$. Define $f^n$ as the ``conjugate'' of $\tilde{g}^n$ as provided by \cref{le:conjugate}, and let similarly $g^n$ be the conjugate of $f^n$ (as provided by the symmetric analogue of \cref{le:conjugate} with interchanged roles of the marginals). By \cref{le:conjugate}, we see that $(f^n,g^n)$ improves upon $(\tilde{f}^n,\tilde{g}^n)$, so that it is again a maximizing sequence. Moreover, $f_n$ and $g_n$ are $\rho$-continuous. And as~\eqref{eq:osc.f} holds for both $f$ and $g$, the inequalities below~\eqref{eq:osc.f} yield the uniform bound
    $\|f_n \oplus g_n\|_\infty \leq t_0 + 5 \|c\|_\infty$. 
    We can shift $(f_n,g_n)$ by a constant such that $f_n(x_0)=0$ at some reference point $x_0\in\Omega$. It follows from \cref{le:conjugate} that $f_n$ and $g_n$ are uniformly bounded. The Arzel\`{a}--Ascoli theorem then yields a subsequential uniform limit $(f_*,g_*)$. Using the uniform convergence, it is easy to see that $(f_*,g_*)$ maximizes $\DUAL(P, Q)$.

    (iv) If $(f_*,g_*)$ maximizes $\DUAL(P, Q)$, taking the directional derivative in an arbitrary direction $(f,g)\in L^\infty(P)\times L^\infty(Q)$, as well as its negative $(-f,-g)$, yields the first-order condition~\eqref{eq:FOC.as}. Conversely, let $(f_*, g_*)$ solve \eqref{eq:FOC.as}. Note that for any fixed $(x,y)$, the function $ \R \ni  s\mapsto  s- \psi(s-c(x,y))$ is concave with gradient  $1- \psi'(s-c(x,y))$.
        Given any functions $(f,g)\in L^\infty(P)\times L^\infty(Q)$, it follows that for $(x,y) \in \Omega \times \Omega'$,
        \begin{align}
        \begin{split}
            f(x) + &g(y) - \psi\left(f(x)+g(y)-c(x,y)\right) \\
             & \leq    f_*(x) + g_*(y) - \psi\left(f_*(x)+g_*(y)-c(x,y)\right)\\
            &\quad+ \{1- \psi'(f_*(x)+g_*(y)-c(x,y)) (f(x)- f_*(x))\} \\
            &\quad +  \{1- \psi'(f_*(x)+g_*(y)-c(x,y)) (g(y)- g_*(y))\}.\label{eq:convexityOfpsi}
        \end{split} 
        \end{align}
        Integrating \eqref{eq:convexityOfpsi}
        with respect to $P\otimes Q$, the final two lines vanish due to \eqref{eq:FOC.as}, which shows that $(f_*,g_*)$ is an optimizer of \eqref{eq:dual1}.

    (v) We have $\psi' \geq 0$ as $\psi$ is increasing, hence $d\pi := \psi'(f_* \oplus g_* - c) d(P \otimes Q)$ is a non-negative measure. %
        The first equation of \eqref{eq:FOC.as} shows that its first marginal is~$P$ and the second equation of \eqref{eq:FOC.as} shows that its second marginal is~$Q$. Hence $\pi\in\Pi(P,Q)$. Using the equality $\varphi(\psi'(z)) = \psi'(z)z - \psi(z)$ with $z := f_* \oplus g_* - c$, integrating, and rearranging, we find
        \[
        \int \big(f_* \oplus g_* - \psi(f_* \oplus g_* - c)\big) d(P\otimes Q)  = \int c \,d\pi + \int \varphi\left(\frac{d\pi}{d(P\otimes Q)}\right) \,d(P\otimes Q).
        \]
        In view of the weak duality $\ROT(P,Q) \geq \DUAL(P,Q)$ shown at the beginning of the proof, this shows that $\pi$ is an optimizer of $\ROT(P, Q)$ and also completes the proof of the strong duality $\ROT(P,Q) = \DUAL(P,Q)$ stated in~(i).

        (vi) Given any solution $(f_*, g_*)$ of \eqref{eq:FOC.as}, applying \cref{le:conjugate} twice yields versions that solve~\eqref{eq:FOC.pw}.

        If $f_*, g_*$ solve \eqref{eq:FOC.pw}, then they are conjugates of one another, hence \cref{le:conjugate} yields the modulus of continuity. As in the proof of (iii), the uniform bound follows from the inequalities below~\eqref{eq:osc.f} using that~\eqref{eq:osc.f} holds for both $f$ and $g$.

    (vii) %
    Let $C$ be such that $f_*\oplus g_* - c\leq C$. 
    By \cref{as:divergence} there are $t_0>0$ and $\delta,\alpha>0$ such that $\psi'(t_0)=1$, and $\psi'(t)<1$ for $t\leq t_0-\delta$, and $\psi''(t)\geq \alpha$ for $t\in [t_0-\delta,C]$. Let $x\in\Omega$. Consider the sets 
    \begin{align*} 
    A&=\{y\in\Omega': f_*(x)+g_*(y)-c(x,y) < t_0 - \delta\},\\
    B&=\{y\in\Omega': t_0 - \delta \leq f_*(x)+g_*(y)-c(x,y) \leq C \}
    \end{align*}
    and note that $B=\Omega'\setminus A$. Set $p:=Q(A)$. As $\psi'$ is nondecreasing, \eqref{eq:FOC.pw} implies
    \begin{align*}    
    1 = \int \psi'(f_*(\cdot)+ g_*(y)-c(\cdot,y))  dQ(y) \leq p \psi'(t_0-\delta) + (1-p) \psi'(C).
    \end{align*}
    This yields the upper bound $p\leq \frac{\psi'(C)-1}{\psi'(C) - \psi'(t_0-\delta)} < 1$ which is uniform in~$x$.  As $\psi''\geq0$, we deduce the uniform lower bound
    \begin{align*}
      \int \psi''(f_*(\cdot)+ g_*(y)-c(\cdot,y))  dQ(y)
      \geq \alpha Q(B) = (1-p) \alpha.
     \end{align*}
    The second claim is shown analogously.
\end{proof}

\subsection{Proofs for \cref{se:linearization}}\label{se:proofsLinearization}

In this subsection we prove the auxiliary \cref{Lemma:aBallInside} on the sections of the set $S=\{(x,y)\in\Omega\times\Omega':\psi''(\xi_*(x,y))> 0\}$.

\begin{proof}[Proof of \cref{Lemma:aBallInside}]
Recall from \cref{pr:ROTprelims}\,(vii) that $S_Q(x)\neq\emptyset$ and $S_P(y)\neq\emptyset$ for all $(x,y)\in\Omega\times\Omega'$. This implies that $x\in O(x):=\cup_{y\in S_Q({x})}S_P({y})$ for all $x\in\Omega$. By continuity of $\zeta:=\psi''\circ \xi_*$, each set $S_P({y})$, and then also the union $O(x)$, is relatively open in $\Omega$. That is, 
\begin{equation}\label{eq:connectedProof1}
    \text{for each $x\in\Omega$ there exists $r>0$ such that $(x+r \mathbb{B})\cap\Omega \subset O(x)$.}
\end{equation}
Consider the set
$$
  A = \{(x,z)\in\Omega\times\Omega : z\in \Omega\setminus O(x)\}.
$$
If $z\in O(x)$, then there exists $y$ such that 
$\zeta(x,y)>0$ and $\zeta(z,y)>0$. By continuity of $\zeta$, for $\tilde{x},\tilde{z}$ sufficiently close to $(x,z)$, we still have $\zeta(\tilde{x},y)>0$ and $\zeta(\tilde{z},y)>0$, showing that $\tilde{z}\in O(\tilde{x})$. This proves that $\{(x,z): z\in O(x)\}$ is open and hence that $A$ is closed. Consider also, for each $r>0$, the closed subset
$$
  A_r = A\cap \{(x,z)\in\Omega\times\Omega : z\in (x+r \overline{\mathbb{B}})\}.
$$
In view of the definition of $A$, the fact~\eqref{eq:connectedProof1} translates to $\cap_{r>0} A_r=\emptyset$. Now the finite intersection property of the compact set $\Omega\times\Omega$ yields that $A_r=\emptyset$ for some $r>0$, which was the claim.
\end{proof}

\end{document}